\documentclass[letter,final]{amsart}

\usepackage[english]{babel}
\usepackage{amsmath,amsfonts,amssymb,amsthm}
\usepackage{graphicx}
\usepackage{psfrag}
\usepackage{latexsym}
\usepackage{enumerate}
\usepackage{xypic}
\usepackage{pstricks}

\newtheorem{theorem}{Theorem}[section]
\newtheorem{corollary}[theorem]{Corollary}
\newtheorem{lemma}[theorem]{Lemma}
\newtheorem{proposition}[theorem]{Proposition}

\theoremstyle{definition}
\newtheorem{definition}[theorem]{Definition}
\newtheorem{remark}[theorem]{Remark}

\newtheorem{example}[theorem]{Example}

\numberwithin{equation}{section}

\newcommand{\placeholder}{{\raise1.5pt\hbox{$\scriptscriptstyle\bullet$}}} % was: \bullet without any rescaling
\newcommand{\indexplural}{} % was: "'s"

\title[The Infrastructure of a Global Field of Arbitrary Unit Rank]
      {The Infrastructure of a Global Field \\ of Arbitrary Unit Rank} 

\author[Felix Fontein]{}
\thanks{This work has been supported in part by the Swiss National Science Foundation under grant
no.~107887.}


\DeclareMathOperator{\Red}{Red}
\DeclareMathOperator{\fRep}{Rep^{\mathit{f}}}
\DeclareMathOperator{\fRepast}{Rep^{\mathit{f}\ast}}
\DeclareMathOperator{\bs}{bs}
\DeclareMathOperator{\gs}{gs}
\DeclareMathOperator{\Norm}{Norm}
\DeclareMathOperator{\reduce}{red}

\newcommand{\N}{\mathbb{N}}
\newcommand{\Z}{\mathbb{Z}}
\newcommand{\Q}{\mathbb{Q}}
\newcommand{\R}{\mathbb{R}}
\newcommand{\C}{\mathbb{C}}
\newcommand{\F}{\mathbb{F}}
\newcommand{\G}{\mathbb{G}}

\newcommand{\calE} {\mathcal{E}}
\newcommand{\calO} {\mathcal{O}}
\newcommand{\calP} {\mathcal{P}}

\newcommand{\frakturFont}[1]{\mathfrak{#1}}
\newcommand{\fraka} {\frakturFont{a}}
\newcommand{\frakb} {\frakturFont{b}}
\newcommand{\frakm} {\frakturFont{m}}

\newcommand{\frakp} {\frakturFont{p}}

\newcommand{\abs}[1]{{\left|{#1}\right|}}
\newcommand{\ggen}[1]{{\left\langle{#1}\right\rangle}}

\DeclareMathOperator{\divisor}{div}
\DeclareMathOperator{\Div}{Div}
\DeclareMathOperator{\Princ}{Princ}
\DeclareMathOperator{\Pic}{Pic}
\DeclareMathOperator{\Id}{Id}
\DeclareMathOperator{\PId}{PId}

\newenvironment{enuma}{\begin{enumerate}[\upshape (a)]}{\end{enumerate}}

\newenvironment{enum1}{\begin{enumerate}[\upshape (1)]}{\end{enumerate}}

\begin{document}

  \maketitle
  
  \centerline{\scshape Felix Fontein}
  \medskip
{\footnotesize
\centerline{Department of Mathematics \&\ Statistics}
\centerline{University of Calgary}
\centerline{2500 University Drive NW}
\centerline{Calgary, Alberta, Canada T2N 1N4}
}

  \begin{abstract}
    In this paper, we show a general way to interpret the infrastructure of a global field of
    arbitrary unit rank. This interpretation generalizes the prior concepts of the giant step
    operation and $f$-representations, and makes it possible to relate the infrastructure to the
    (Arakelov) divisor class group of the global field. In the case of global function fields, we
    present results that establish that effective implementation of the presented methods is indeed
    possible, and we show how Shanks' baby-step giant-step method can be generalized to this
    situation.
  \end{abstract}

  \section{Introduction}
  
  The infrastructure of a global field, i.e.\ of a number field or a function field over a finite
  field, is a group-like algebraic structure. It is a crucial ingredient in the computation of the
  regulator, a system of fundamental units, and the order and structure of the ideal class group. In
  the case of a one-dimensional infrastructure, which occurs in in fields of unit rank~one, this
  group-like structure was first used by D.~Shanks to compute the regulator of a real quadratic
  number field via a baby-step giant-step algorithm.
  
  In this paper, we present a framework of infrastructure that unifies number fields and function
  fields. The crucial tool to accomplish this are $f$-representations; these represent a group well
  suited for computations into which the infrastructure embeds. Using $f$-representations, we obtain
  giant steps, which are an important tool in algorithms of baby-step giant-step type. We establish
  that $f$-representations require little storage and lend themselves very well to computation. They
  can be efficiently used for determining a system of fundamental units of a global function
  field. We provide evidence for this by presenting preliminary implementation results as well as
  non-trivial numerical examples.
  
  The idea behind $f$-representations was described in \cite{ff-pohlighellman} in the
  one-di\-men\-sion\-al case, i.e.\ for infrastructures obtained from global fields of unit
  rank~one. The concept of $f$-representations goes back to $(f, p)$-representations, which were
  introduced in the context of cryptography in real quadratic number fields by D.~H\"uhnlein and
  S.~Paulus \cite{huehnlein-paulus} and M.~J.~Jacobson~Jr., R.~Scheidler and H.~C.~Williams
  \cite{JSW-RQFKE}.
  
  Infrastructures of number fields and, more recently, of function fields have been studied for some
  time. Their investigation has its roots in C.~F.~Gau\ss' study of the composition of binary
  quadratic forms, as well as J.-L.~Lagrange's continued fraction algorithm. The infrastructure
  first appeared explicitly in the context of generalizing continued fraction expansion. In his PhD
  dissertation, G.~Vorono\u\i\ found a generalization of continued fraction expansion by minima of
  lattices and formulated an algorithm to find a system of fundamental units of a cubic number field
  \cite{delonefaddeev}.
  
  The set of minima of a global field was studied, for example in
  \cite{bergmannnetze,hellegouarchpaysant2}, and was used for computing fundamental units in number
  fields, for example in \cite{pohst-zassenhaus-fundamentalunits1, steiner-units, appelgateonishi,
  pohst-zassenhaus-fundamentalunits2,pohst-zassenhaus-fundamentalunits3}. J.~A.~Buchmann generalized
  Vorono\u\i's algorithm to number fields of unit rank one and two
  \cite{genvoronoiIuII}. Subsequently, he presented a generalization of Lagrange's algorithm for
  computing fundamental units in arbitrary number fields in $\calO(R \cdot
  \abs{\Delta}^\varepsilon)$ binary operations; here $\varepsilon > 0$ arbitrary, $R$ is the
  regulator and $\Delta$ the discriminant of the number field \cite{genlagrange}. Note that $R =
  \calO(\abs{\Delta}^{1/2 + \varepsilon})$ for any~$\varepsilon > 0$.
  
  In 1972, D.~Shanks \cite{shanks-infra} discovered that the principal infrastructure of a real
  quadratic number field supports a group-like structure. With every element of the principal
  infrastructure is associated a distance which imposes an ordering on this set. The infrastructure
  supports two operations: a \emph{baby step}, which proceeds cyclically from one element to the
  next in this ordering, and a \emph{giant step}, which is akin to multiplication in a cyclic group
  and under which distances behave almost additively. As a result, the principal infrastructure is
  almost an abelian group under giant steps that only slightly fails associativity. Using this
  group-like behavior, Shanks was able to compute the regulator and therefore the absolute value of
  a fundamental unit of a real quadratic number field in $\calO(\sqrt{R})$ steps instead of the
  $\calO(R)$ steps required by the classical algorithm of Lagrange. Note that writing down a system
  of fundamental units requires $\calO(R)$ binary operations, whence no algorithm can compute a
  system of fundamental units in time faster than $\calO(R)$. However, the logarithm of an absolute
  value of a fundamental unit can be computed faster. Shanks' method was further analyzed and
  refined by H.~W.~Lenstra, R.~Schoof, H.~C.~Williams and M.~C.~Wunderlich in
  \cite{lenstra-infrastructure, schoof-infra, williams-contfract-numbtheor-compus,
  williams-wunderlich}, and finally generalized to all number fields of unit rank~one by Buchmann
  and Williams \cite{buchmannwilliams-infrastructure}.
  
  Shanks' method was first extended to function fields in works of A.~Stein and H.~G.~Zimmer
  \cite{stein-da, stein-zimmer}, Stein and Williams \cite{stein-williams-regulator,
  stein-williams-regulator2} and Scheidler and Stein \cite{scheidler-stein-purelyunitrank1,
  scheidler-infrastructurepurelycubic}. The relationship between the infrastructure in real elliptic
  and hyperelliptic function fields and the divisor class group in their imaginary counterparts was
  investigated by Stein in \cite{stein-realelliptic}, and by S.~Paulus and H.-G.~R\"uck in
  \cite{paulus-rueck}.
  
  Shanks' discovery of the infrastructure also led to a number of cryptographic applications.  The
  first of these was a Diffie-Hellman-like key exchange protocol described by Buchmann and Williams,
  and later by Scheidler, Buchmann and Williams in
  \cite{buchmann-williams-qrkeyexchange,scheidler-buchmann-williams-qrkeyexchange}. This was
  extended in several ways, and additional encryption and signature schemes were proposed; some of
  these are described in \cite{biehl-buchmann-thiel, SSW-KEiRQCFF, JSW-RQFKE2, JSS-CPoRHC}. The
  security of these systems is argued to be based on the hardness of computing distances or
  computing the regulator; the hardness of these problems is analyzed, for example, in
  \cite{mueller-stein-thiel, jacobson-phd, maurer-phd, vollmer-phd}.
  
  All efficient algorithms and cryptosystems based on the infrastructure crucially require the giant
  step operation. This raises the question of whether a giant step can be defined and used
  efficiently in all global fields, not just of unit rank one. In the number field case, Buchmann
  showed in his habilitation thesis \cite{buchmann-habil} that there is in fact such a giant step,
  and that this giant step can be used to compute the absolute values of fundamental units in
  $\calO(\sqrt{R} \cdot \abs{\Delta}^\varepsilon)$ binary operations. Unfortunately, this algorithm
  was only published in Buchmann's thesis, which was written in German and is not easily accessible.
  Later, Schoof presented a modern treatment of the general number field case using Arakelov divisor
  theory \cite{schoofArakelov}. This is so far the most general treatment of infrastructure. It
  includes the concept of a giant step, even though Schoof does not give a baby-step giant-step
  algorithm like Buchmann's. Both Buchmann's and Schoof's giant steps rely on a simple reduction
  strategy: the infrastructure is in both cases a subset of the set of fractional ideals, whose
  elements are called ``reduced ideals'', and the giant step roughly corresponds to multiplying two
  such ideals. The result is in general not inside this set, but after finding a ``short'' element
  in the product and dividing by it, the resulting ideal will lie in this set. This process of
  chosing the short element is called \emph{reduction}.
  
  In this paper, we present for the first time a unified treatment of number fields and function
  fields and define infrastructure for any unit rank. Moreover, we provide a connection between the
  infrastructure and the (Arakelov) divisor class group and relate the arithmetic in these two
  objects. The key point is a more sophisticated reduction strategy, mimicking the reduction
  described by F.~He\ss\ for arithmetic in the divisor class groups of global function fields
  \cite{hessRR}. For that, we have to use a slightly different embedding of the reduced ideals into
  the Arakelov divisor class group than the one used by Schoof. We also do not use the oriented
  Arakelov divisor class group, but instead an equivalence relation on reduced ideals in case when
  there is no real embedding of the number field. This allows us to unify arithmetic in the
  (Arakelov) divisor class group of both number fields and function fields. Moreover, in contrast to
  Schoof's work, we ``parameterize'' the Arakelov divisor class group using equivalence classes of
  reduced divisors together with a finite set of real numbers, and can describe explicit arithmetic
  using this representation. This parameterization generalizes the aforementioned result by Paulus
  and R\"uck \cite{paulus-rueck} on hyperelliptic function fields, and, since it extends He\ss'
  approach, it also generalizes known arithmetic in imaginary hyperelliptic and superelliptic
  function fields \cite{hehcc, galbraith-paulus-smart}.
  
  This paper is organized as follows. We begin by giving an overview of the arithmetic in number
  fields and function fields in Section~\ref{fnf-intro}. Then we will provide an abstract treatment
  of one-dimensional infrastructures, based on \cite{ff-pohlighellman}, in
  Section~\ref{odi-intro}. We discuss how an abstract $n$-dimensional infrastructure can be defined
  in Section~\ref{ndiminfra}. We go on to describe reduced ideals in Section~\ref{reducedideals}. In
  Section~\ref{geninfra}, we show how to obtain an infrastructure in any global field satisfying the
  definition given in Section~\ref{ndiminfra}. We explore the relationship between the
  infrastructure and the divisor class group in Section~\ref{relationtodcg}. In
  Section~\ref{computationsfprep}, we show that $f$-representations can be used to effectively
  perform computation of fundamental units in function fields. Finally, we provide concluding
  remarks and pose some pertinent open questions in Section~\ref{conclusion}.
  
  \section{Arithmetic in Function Fields and Number Fields}
  \label{fnf-intro}

  Let $K$ be a global field, i.e., either a function field over a finite field of constants~$k$, or
  an algebraic number field. In the latter case, denote by $k^*$ the roots of unity of $K$ and set
  $k := k^* \cup \{ 0 \}$.
  
  If $K$ is an algebraic function field, we assume that $k$ is the exact field of constants of
  $K$. Let $x \in K$ be transcendental over $k$.\footnote{Note that we do not assume that $K / k(x)$
  is separable.} Let $\calO_K$ denote the integral closure of $k[x]$ in $K$ and $S$ the set of
  places of $K / k$ which do not correspond to prime ideals of $\calO_K$, i.e.\ the places of $K$
  lying over the infinite place of $k(x)$. Note that for any non-empty finite choice of $S$, one can
  find such an $x$ so that $S$ is the set of places lying over the infinite place of $k(x)$. We
  assume that $x$ and $S$ are fixed throughout this paper. In the number field case, let $\calO_K$
  denote the integral closure of $\Z$ in $K$ and $S$ the set of all archimedean places of $K$. In
  both cases, we denote by $\calP_K$ the set of all places of $K$. If $\frakp \in \calP_K$ is a
  non-archimedean place, let $\nu_\frakp$ be its normalized discrete valuation, $\calO_\frakp$ its
  valuation ring and $\frakm_\frakp$ its valuation ideal. All places in $S$ are called
  \emph{infinite}, and all others \emph{finite}. All finite places are non-archimedean.
  
  In the function field case, the group of divisors~$\Div(K)$ is the free abelian group generated by
  $\calP_K$. For a divisor $D = \sum_{\frakp \in \calP_K} n_\frakp \frakp$, the degree is defined as
  $\deg D := \sum_{\frakp \in \calP_K} n_\frakp \deg \frakp$. The divisors of degree~zero form a
  subgroup of $\Div(K)$, denoted by $\Div^0(K)$. For an element $f \in K^*$, the principal divisor
  of $f$ is defined by $(f) := \sum_{\frakp \in \calP_K} \nu_\frakp(f) \frakp \in \Div^0(K)$; the
  set of all such divisors forms the group $\Princ(K)$, and the quotient group $\Pic^0(K) :=
  \Div^0(K) / \Princ(K)$ is called the (degree zero) divisor class group of $K$. Moreover, we have
  the quotient~$\Pic(K) := \Div(K) / \Princ(K)$ together with the exact sequence $\xymatrix{ 0
    \ar[r] & \Pic^0(K) \ar[r] & \Pic(K) \ar[r]^{\;\; \deg} & \Z }$. Note that the last map (after
  restricting the codomain to the image) splits across this exact sequence, whence we have $\Pic(K)
  \cong \Pic^0(K) \times \Z$.
  
  In the number field case, the group of divisors~$\Div(K)$ is the direct product of the free
  abelian group generated by all places outside $S$ and the abelian group $\R^S$ of all tuples
  $(n_\frakp)_{\frakp \in S}$ of real numbers with pointwise addition. We write elements
  $(n_\frakp)_{\frakp \in S} \in \R^S$ additively as $\sum_{\frakp \in S} n_\frakp \frakp$. For
  $\frakp \in S$, let $\sigma : K \to \C$ be a corresponding embedding; define $\deg \frakp := 1$ if
  $\sigma(K) \subseteq \R$ and $\deg \frakp := 2$ elsewhere. Also define $\nu_\frakp(f) := -\log
  \abs{\sigma(f)}$ for any $f \in K^*$. If $\frakp$ is a finite place, i.e.\ $\frakp \not\in S$,
  define $\deg \frakp := \log \abs{\calO_\frakp / \frakm_\frakp}$. Here, $\log$ denotes the natural
  logarithm. The definition of the degree of a divisor and of a principal divisor is analogous to
  the function field case, as is the definition of $\Pic^0(K)$ and $\Pic(K)$, and we get $\Pic(K)
  \cong \Pic^0(K) \times \R$ in the same way as above.
  
  If $D = \sum_{\frakp \in \calP_K} n_\frakp \frakp$ is a divisor, the places $\frakp \in \calP_K$
  with $n_\frakp \neq 0$ form the \emph{support} of $D$. If $K$ is a global function field, let $q =
  \abs{k} < \infty$. For non-global function fields, let $q > 1$ be arbitrary. For number fields,
  let $q = e = \exp(1)$. Then define the absolute value with respect to a place $\frakp \in \calP_K$
  by $\abs{f}_\frakp := q^{-\nu_\frakp(f) \deg \frakp}$ for $f \in K^*$ and $\abs{0}_\frakp :=
  0$. The fact that principal divisors have degree~zero translates to the product formula
  $\prod_{\frakp \in \calP_K} \abs{f}_\frakp = 1$ for $f \in K^*$.
  
  In both number fields and function fields, a finitely generated $\calO_K$-submodule of $K$ is
  called a fractional ideal. Throughout this paper, we will often say ``ideal'' when we mean
  ``non-zero fractional ideal''. The set of non-zero fractional ideals~$\Id(\calO_K)$ forms a free
  abelian group under multiplication, with the set of non-zero prime ideals of $\calO_K$ as a
  basis. These prime ideals correspond to the places of $K$ outside $S$: if $\frakp$ is such a
  place, $\frakm_\frakp \cap \calO_K$ is the corresponding prime ideal of $\calO_K$. Moreover, we
  have a natural homomorphism $\Div(K) \to \Id(\calO_K)$ defined by $\sum n_\frakp \frakp \mapsto
  \prod_{\frakp \not\in S} (\frakm_\frakp \cap \calO_K)^{-n_\frakp}$. This homomorphism extends to a
  map $\Pic^0(K) \to \Pic(\calO_K)$, where $\Pic(\calO_K) := \Id(\calO_K) / \Princ(\calO_K)$ is the
  ideal class group of $\calO_K$, i.e.\ the quotient of $\Id(\calO_K)$ with the subgroup
  $\Princ(\calO_K) = \{ \frac{1}{f} \calO_K \mid f \in K^* \}$ of non-zero principal fractional
  ideals.
  
  Note that forming principal divisors or principal ideals give homomorphisms $K^* \to \Princ(K)
  \subseteq \Div^0(K)$, $f \mapsto (f)$ and $K^* \to \Princ(\calO_K) \subseteq \Id(\calO_K)$, $f
  \mapsto \frac{1}{f} \calO_K$. Finally, denote by $\Div_\infty^0(K)$ the set of divisors in
  $\Div^0(K)$ that are only supported at places in $S$. All the aforementioned maps give rise to the
  following commuting diagram with exact rows and columns: \[ \xymatrix@R-0.5cm{ & 0 \ar[d] & 0
  \ar[d] & 0 \ar[d] & \\ 0 \ar[r] & \calO_K^* / k^* \ar[r] \ar[d] & \Div_\infty^0(K) \ar[r] \ar[d] &
  T \ar[r] \ar[d] & 0 \\ 0 \ar[r] & K^* / k^* \ar[r] \ar[d] & \Div^0(K) \ar[r] \ar[d] & \Pic^0(K)
  \ar[r] \ar[d] & 0 \\ 0 \ar[r] & K^* / \calO_K^* \ar[r] \ar[d] & \Id(\calO_K) \ar[r] \ar[d] &
  \Pic(\calO_K) \ar[r] \ar[d] & 0 \\ & 0 & H \ar[r]^{\cong} \ar[d] & H' \ar[d] & \\ & & 0 & 0 & } \]
  Here, $T$, $H$ and $H'$ are suitable groups that are discussed in more detail below.
  
  If $K$ is a number field, $\Div_\infty^0(K) \cong \R^{\abs{S} - 1}$, the image of $\calO_K^* /
  k^*$ is a lattice of full rank in $\R^{\abs{S} - 1}$ and hence $T$ is an $(\abs{S} -
  1)$-dimensional torus. Moreover, $H = 0$ and $H' = 0$. If $K$ is a function field, then
  $\Div_\infty^0(K) \cong \Z^{\abs{S} - 1}$. If $k$ is finite, then $T$ is finite by an analogue of
  Dirichlet's Unit Theorem \cite[p.~243, Proposition~14.2]{rosen}. In case $k$ is infinite, $T$ can
  be finite or infinite, and both possibilities occur; see \cite[Section~4]{hellegouarchpaysant2}
  for examples with $k = \Q$.  We have $H = 0 = H'$ if and only if $(\deg \frakp \mid \frakp \in S)
  = (\deg \frakp \mid \frakp \in \calP_K)$, as $H \cong (\deg \frakp \mid \frakp \in \calP_K) /
  (\deg \frakp \mid \frakp \in S)$. Here, $(\deg \frakp \mid \frakp \in S)$ is the ideal in $\Z$
  generated by $\{ \deg \frakp \mid \frakp \in S \}$; $(\deg \frakp \mid \frakp \in S)$ is defined
  analogously.
  
  For both number fields and function fields, the rank of $\calO_K^* / k^*$ is called the \emph{unit
  rank} of $K$. In case $K$ is a number field or $T$ is finite, the rank equals $\abs{S} - 1$. Note
  that we assumed $x$ to be fixed in the function field case. If the unit rank equals $n = \abs{S} -
  1$, let $\frakp_1, \dots, \frakp_n \in S$ be $n$ distinct places, and $\varepsilon_1, \dots,
  \varepsilon_n$ a system of fundamental units of $\calO_K$, i.e.\ a set of units whose residue
  classes in $\calO_K^*/k^*$ are a basis of $\calO_K^*/k^*$. Define \[ R := \abs{\det \Bigl(
  \nu_{\frakp_i}(\varepsilon_j) \deg \frakp_i \Bigr)_{1 \le i, j \le n }} \in \R_{\ge 0}; \] this
  number is the \emph{regulator} of $K$ (after fixing $x$ in the function field case) and is
  independent of the choice of the $\frakp_i$\indexplural\ and of the choice of the
  $\varepsilon_j$\indexplural.
  
  \section{One-Dimensional Infrastructures}
  \label{odi-intro}
  
  A one-dimensional infrastructure can be interpreted as a circle with a finite set of points on
  it. This interpretation goes back to Lenstra's work in \cite{lenstra-infrastructure}. See also
  \cite{ff-pohlighellman} for an earlier treatment of (abstract) one-dimensional infrastructures.
  
  \begin{definition}\label{odi-definition}
    A \emph{one-dimensional infrastructure}~$(X, d)$ of \emph{circumference~$R > 0$} is a finite set
    $X \neq \emptyset$ together with an injective map~$d : X \to \R/R\Z$.
  \end{definition}
  
  This can be visualized as follows; see also Figure~\ref{fig:circle}(a). One can interpret $\R/R\Z$
  as a circle of circumference~$R$, with a fixed point $0 \in \R/R\Z$. Then $d(X)$ is a finite set
  of points on this circle, and for every $x \in X$, the residue class $d(x)$ can be interpreted as
  the \emph{distance} of the point $d(x)$ on the circle to $0$.
  
  \begin{figure}[b]
    \begin{center}%
      \begin{tabular}{ccccc}
        \psfrag{0}{\footnotesize$0$}%
        \includegraphics[height=2.5cm]{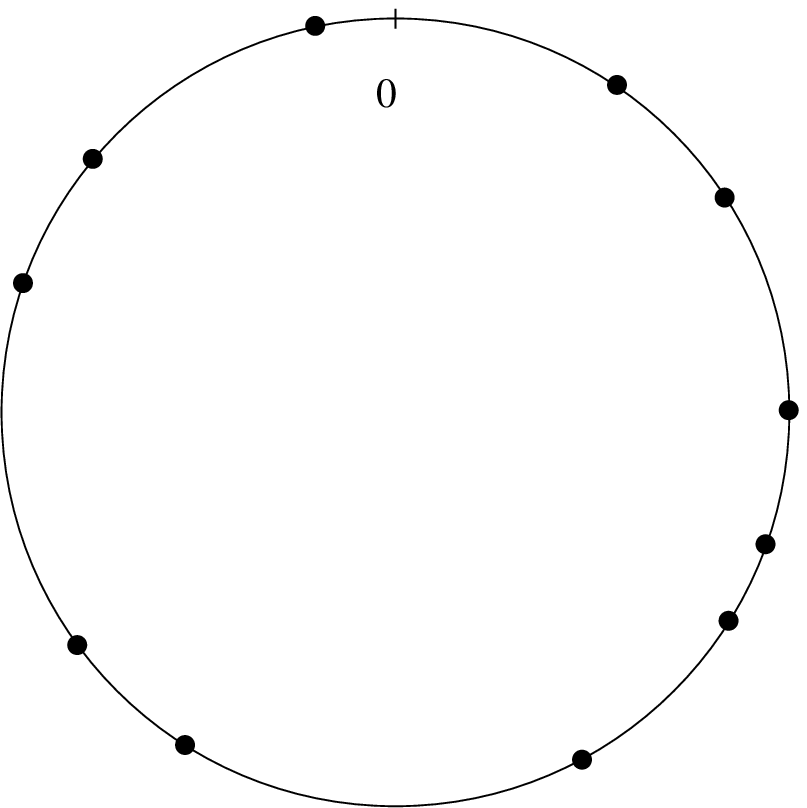}%
      &%
        \psfrag{s}{\footnotesize$s$}\psfrag{t}{\footnotesize$t$}\psfrag{[s,t]}{\footnotesize$[s,t]$}%
        \psfrag{0}{\footnotesize$0$}%
        \includegraphics[height=2.5cm]{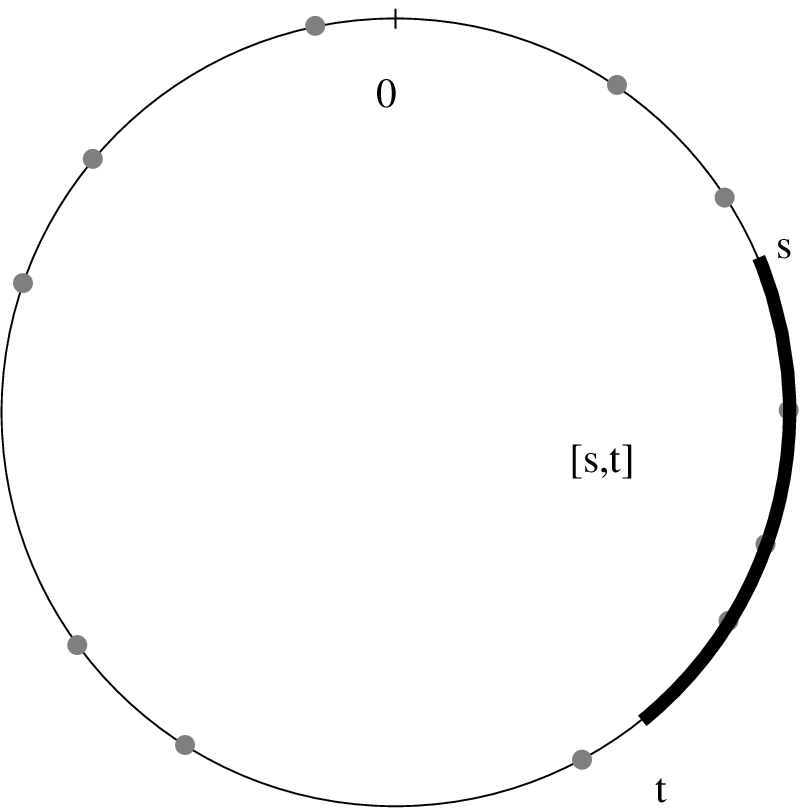}%
      &%
        \psfrag{x}{$x$}\psfrag{bs(x)}{\footnotesize$\bs(x)$}\psfrag{[x,bs(x)]}{\footnotesize$[x,\bs(x)]$}%
        \psfrag{0}{\footnotesize$0$}%
        \includegraphics[height=2.5cm]{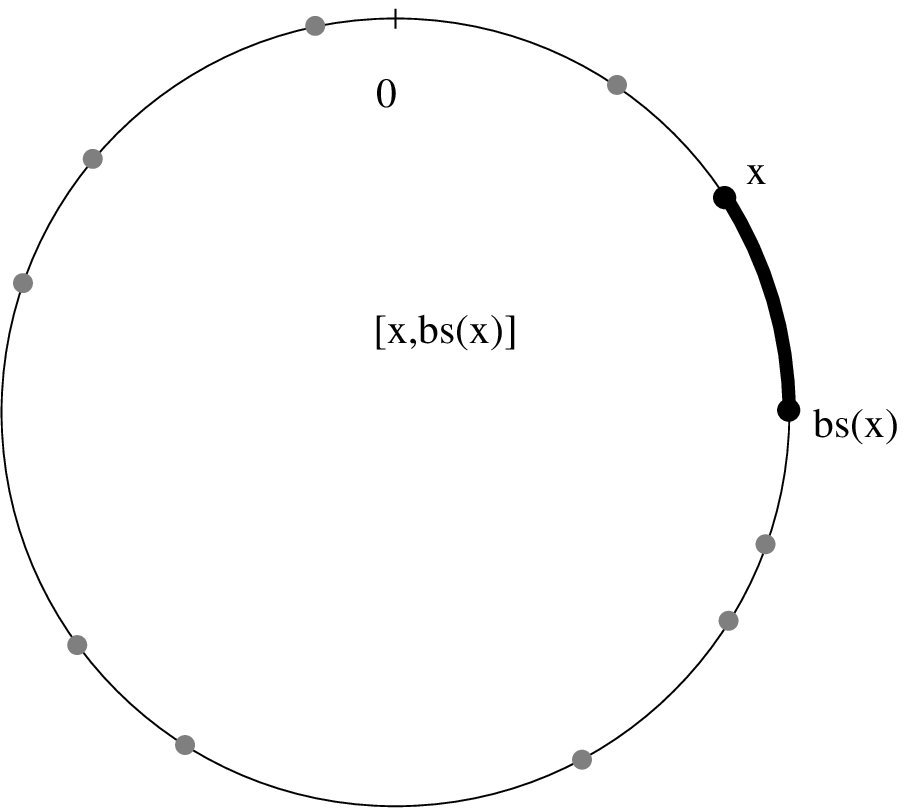}%
      & &%
        \psfrag{x}{\footnotesize$x$}\psfrag{y}{\footnotesize$y$}\psfrag{0}{\footnotesize$0$}%
        \psfrag{d(x)+d(y)}{\footnotesize$d(x){+}d(y)$}\psfrag{gs(x,y)}{\footnotesize$\gs(x,y)$}%
        \includegraphics[height=2.5cm]{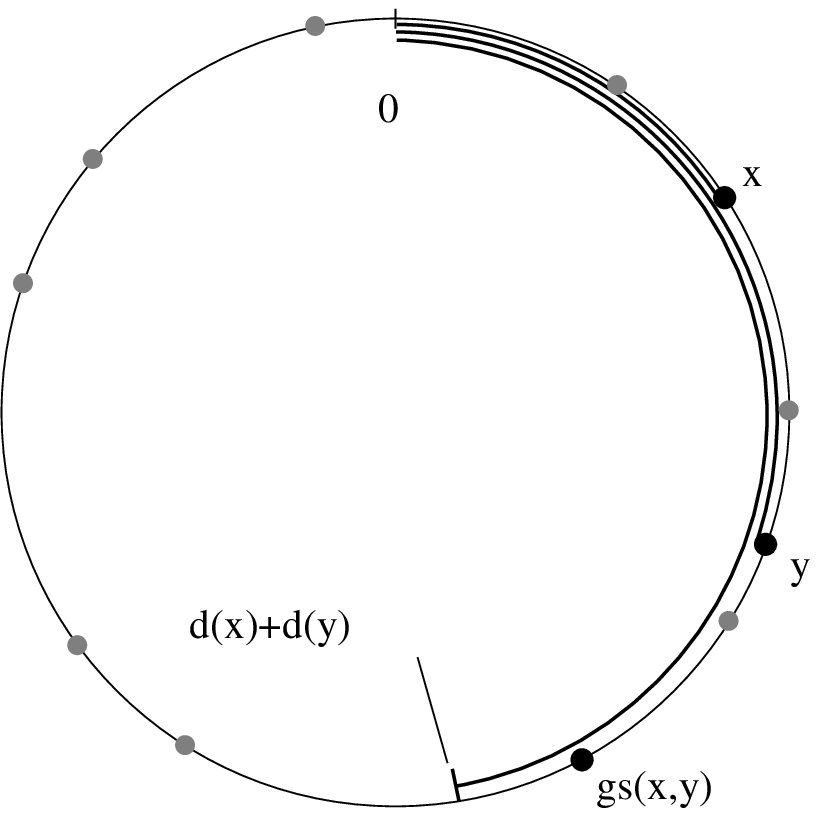}%
      \\ (a) & (b) & (c) & & (d)
      \end{tabular}
      \caption{Illustrating a one-dimensional infrastructure using a circle}%
      \label{fig:circle}
    \end{center}
  \end{figure}
  
  The infrastructure essentially offers two operations:
  \begin{itemize}
    \item \emph{baby steps}: given $x \in X$, the baby step $\bs(x)$ denotes the preimage of the
    element in $d(X)$ on the circle ``following'' $d(x)$;
    \item \emph{giant steps}: given $x, y \in X$, the giant step $\gs(x, y)$ denotes the preimage of
    the element in $d(X)$ on the circle ``before'' $d(x) + d(y)$.
  \end{itemize}
  
  We want to make this more precise. If $\abs{X} = 1$, there is only one way to define $\bs : X \to
  X$ and $\gs : X \times X \to X$ by $\bs(x) = x$, $\gs(x, x) = x$ if $X = \{ x \}$. If $\abs{X} >
  1$, then we can define these two maps as follows.
 
  For $s = a + R\Z$ and $t = b + R\Z$ with $a \le b < a + R$, we denote by $[s, t]$ the set $\{ x +
  R \Z \mid a \le x \le b \}$. If we interpret $\R/R\Z$ as a circle, $[s, t]$ will be the circle
  segment starting at $s$ and ending at $t$ in positive direction. See also
  Figure~\ref{fig:circle}(b).
  
  Then for $x \in X$, we can define $\bs(x)$ as the unique element of $X \setminus \{ x \}$
  satisfying \[ \{ d(x), d(\bs(x)) \} = d(X) \cap [d(x), d(\bs(x))], \] i.e.\ the only two points in
  $d(X)$ lying on the circle segment $[d(x), d(\bs(x))]$ are $d(x)$ and $d(\bs(x))$; see
  Figure~\ref{fig:circle}(c). For $x, y \in X$, we can define $\gs(x, y)$ as the unique element of
  $X$ satisfying \[ \{ d(\gs(x, y)) \} = d(X) \cap [d(\gs(x, y)), d(x) + d(y)], \] i.e.\ the only
  point in $d(X)$ lying on the circle segment $[d(\gs(x, y)), d(x) + d(y)]$ is $d(\gs(x, y))$; see
  Figure~\ref{fig:circle}(d).
  
  The simplest example of one-dimensional infrastructures, which is nevertheless important, is given
  by finite cyclic groups:
  
  \begin{example}
    \label{cyclic-group-example}
    Let $G = \ggen{g}$ be a finite cyclic group of order~$R$. Then we have a canonical
    isomorphism~$\varphi : \Z/R\Z \to G$, $n \mapsto g^n$. Concatenating its inverse with the
    inclusion $\Z/R\Z \subset \R/R\Z$, we obtain an injective map $d : G \to \R/R\Z$, making $(G,
    d)$ a one-dimensional infrastructure. This map is the \emph{discrete logarithm} map with
    base~$g$, i.e.\ it satisfies $g^{d(h)} = h$ for every $h \in G$.
    
    Let $h \in G$ and $d(h) = n + R\Z$. Then for $h' = g^{n'}$ with $n \le n' < n + R$, we have
    $[d(h), d(h')] \cap d(X) = \{ d(g^n), d(g^{n+1}), \dots, d(g^{n'-1}), d(g^{n'}) \}$. This shows
    that if this set contains exactly two elements, then $n' = n + 1$. But this translates to
    $\bs(h) = g h$, so baby steps on $G$ are simply multiplication by the generator $g$ of $G$.
    
    Similarly, if $h = g^n$ and $h' = g^{n'}$, we see that $d(X) \cap [d(g^m), d(h) + d(h')] = \{
    d(g^m), d(g^{m+1}), \dots, d(g^{n + n' - 1}), d(g^{n + n'}) \}$ if $m \le n + n' < m + R$. This
    shows that $\gs(h, h') = g^{n + n'} = h h'$, so giant steps on $G$ amount to group
    multiplication.
  \end{example}
  
  In this paper, we will concentrate on giant steps as they are needed to obtain algorithms of
  square root type, which compute the absolute values of a system of fundamental units in
  $\calO(\sqrt{R})$ infrastructure operations, where $R$ is the regulator of the field.
  
  The giant step is a binary operation on the finite set $X$ which is not necessarily
  associative. For certain applications, such as using the infrastructure in cryptography, one is
  interested in having associative operations: the Diffie-Hellman key exchange protocol depends on
  the fact that $(x^a)^b = (x^b)^a$ for all $a, b \in \N$. More precisely, it is not obvious how to
  define $x^a$ without having an associative operation.
  
  In the infrastructure case, one could define $x^a$ as an element $y \in X$ such that $a \cdot d(x)
  \approx d(y)$. But then it is not necessarily true that $(x^a)^b$ is equal to $(x^b)^a$. One only
  knows that $d((x^a)^b) \approx a \cdot b \cdot d(x) \approx d((x^b)^a)$, but the error here can be
  up to $a$ or $b$ times larger than in $a \cdot d(x) \approx d(y)$. In
  Example~\ref{cyclic-group-example} above, where we start with a finite cyclic group~$G$, this
  error is always $0$ since $G$ is of course associative, and we recover the original Diffie-Hellman
  key exchange protocol, whose security is based on the fact that computing the map $d : G \to
  \R/R\Z$ is hard for random elements of $G$.
  
  Note that while the giant step operation is in general not associative, it is \emph{almost}
  associative: it is so up to a ``small error'', which can be bounded by $d_{\max} := \max\{
  d(\bs(x)) - d(x) \mid x \in X \}$, where we identify $d(\bs(x)) - d(x)$ with the smallest
  non-negative real number lying in the residue class modulo $R$. Namely, we have \[ d(\gs(x, y)) =
  d(x) + d(y) - \varepsilon_{x,y} \qquad \text{with } 0 \le \varepsilon_{x,y} < d_{\max}. \] In
  terms of Figure~\ref{fig:circle}, $d_{\max}$ is the maximal distance between two adjacent points
  on the circle. In Example~\ref{cyclic-group-example} above, we have $d_{\max} = 1$, even though
  the error~$\varepsilon_{x,y}$ will always be zero.
  
  One can ask whether this gap towards an associative operation can be closed. One solution is to
  embed infrastructures into groups. Obviously, $\R/R\Z$ is a group under addition. Unfortunately,
  as seen in Example~\ref{cyclic-group-example}, the embedding $d$ is in general not very helpful,
  since it is often hard to evaluate; in the example, evaluating it is equivalent to compute a
  discrete logarithm, which, depending on the group~$G$, can be a hard problem; see for example
  \cite{hehcc}. We want a group suitable for effective computations, into which $X$ embeds by an
  easily computable embedding. In order to achieve that, we require $f$-representations:
  
  \begin{figure}[t]
    \begin{center}%
        \psfrag{x}{\footnotesize$x$}\psfrag{t}{\footnotesize$t$}\psfrag{s}{\footnotesize$s$}%
        \includegraphics[height=2.5cm]{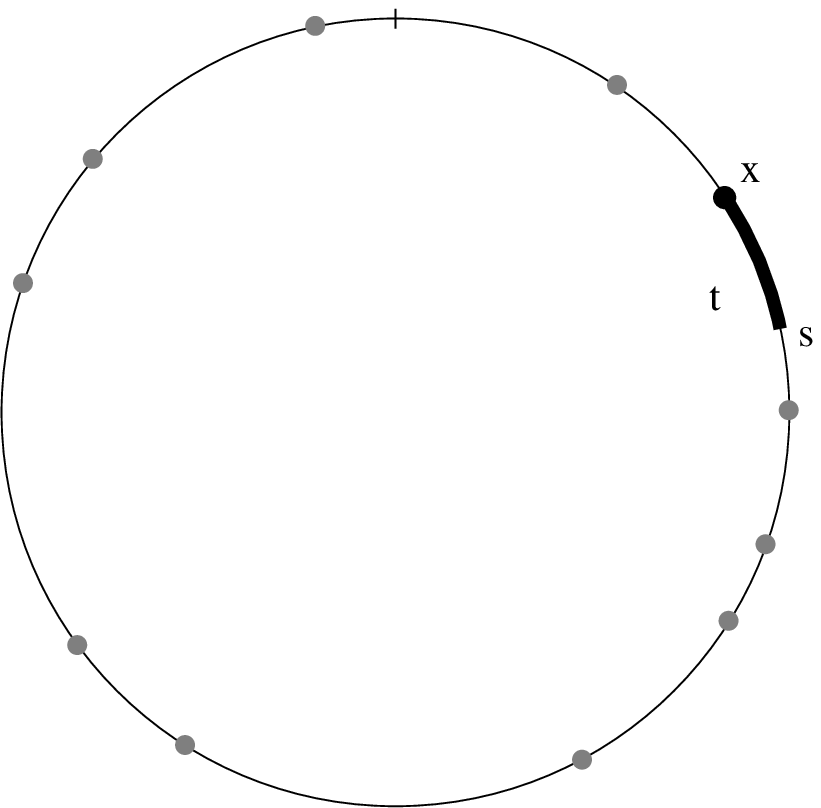}%
      \quad\;%
        \psfrag{x}{\footnotesize$x$}\psfrag{t}{\footnotesize$t$}\psfrag{s}{\footnotesize$s$}%
        \includegraphics[height=2.5cm]{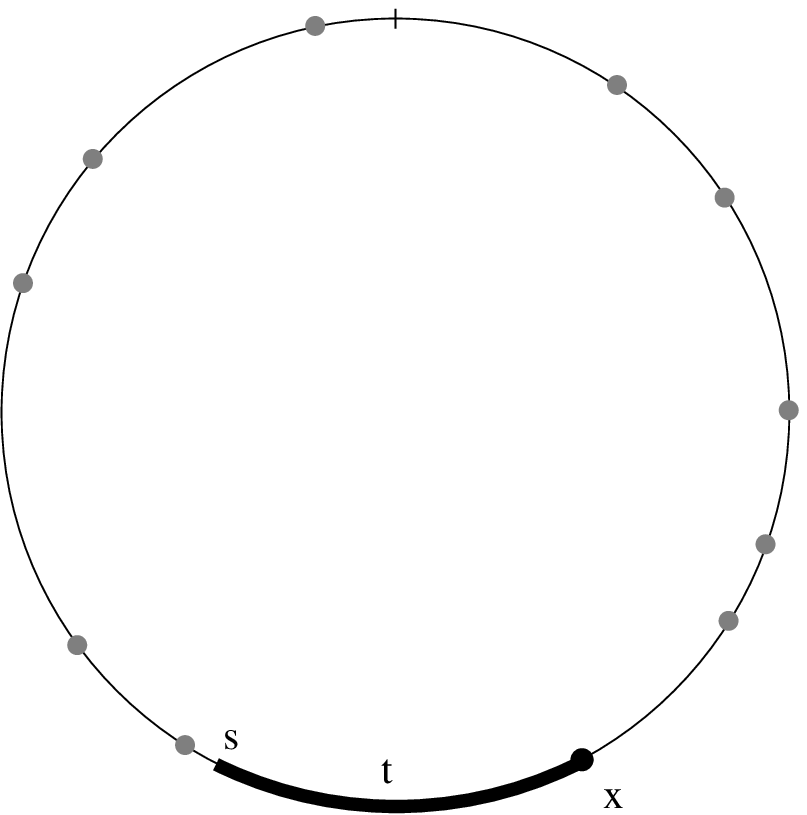}%
      \quad\;%
        \psfrag{x}{\footnotesize$x$}\psfrag{t}{\footnotesize$t$}\psfrag{s}{\footnotesize$s$}%
        \includegraphics[height=2.5cm]{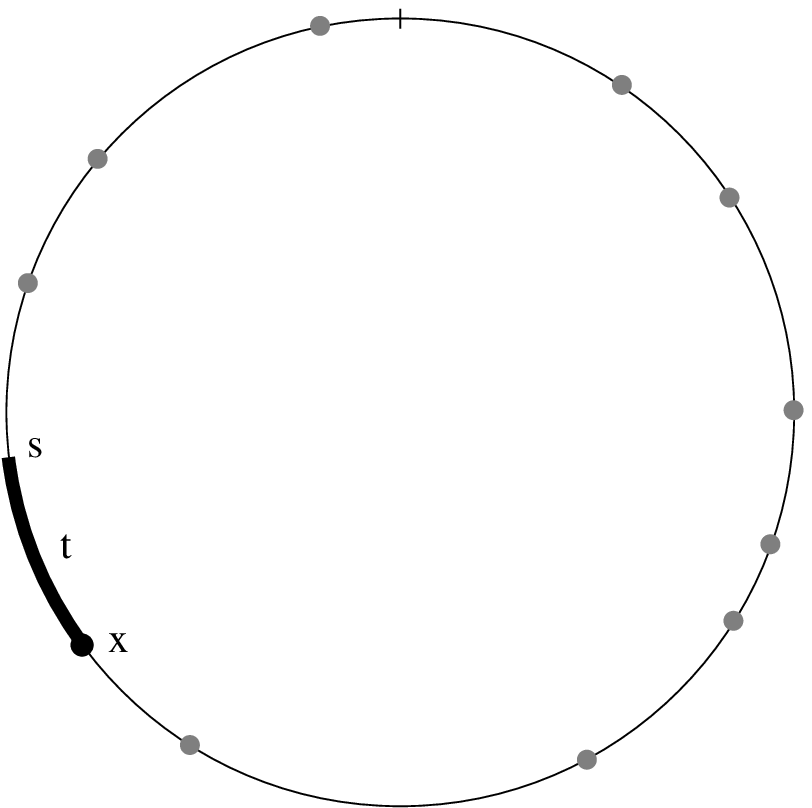}%
      \quad\;%
        \psfrag{x}{\footnotesize$x$}\psfrag{t}{\footnotesize$t$}\psfrag{s}{\footnotesize$s$}%
        \includegraphics[height=2.5cm]{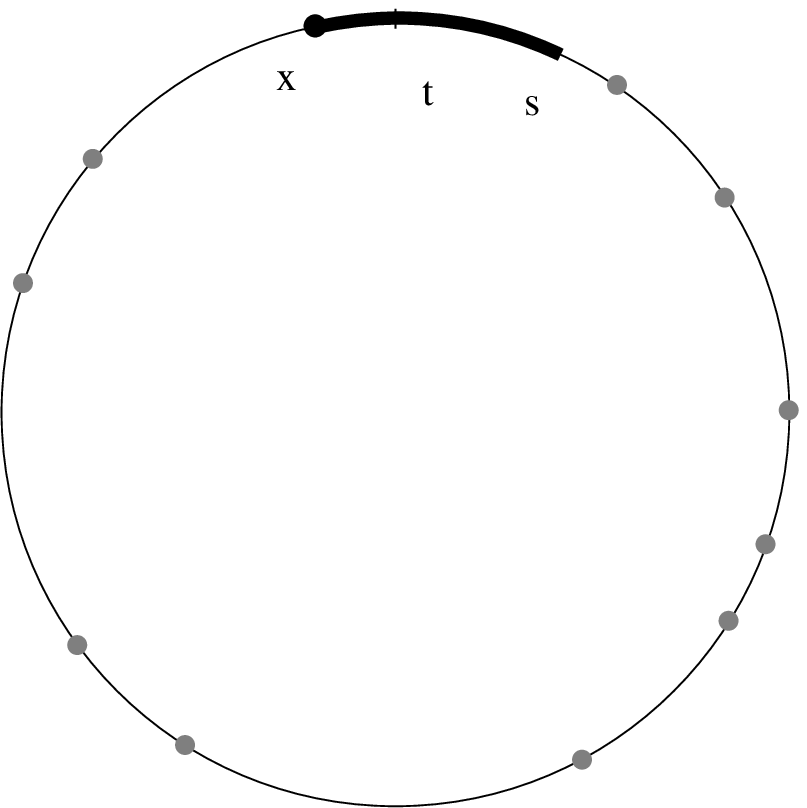}%
      \caption{Illustrating $f$-representations in a one-dimensional infrastructure}%
      \label{fig:circle-frep}
    \end{center}
  \end{figure}
  
  \begin{definition}
    \label{onedimfRepDef}
    An element $(x, t) \in X \times \R$ is called an \emph{$f$-representation} if $0 \le t < R$ and
    \[ \{ d(x) \} = [d(x), d(x) + t] \cap d(X). \] Denote the set of all $f$-representations by
    $\fRep(X, d)$.
  \end{definition}
  
  The $f$-representation $(x, t)$ \emph{represents} the element $s := d(x) + t \in \R/R\Z$. The
  condition on $t$ implies that $t$ is minimal for such a representation: it is the smallest
  distance from the point $s = d(x) + t$ on the circle backwards to a point in $d(X)$ (namely
  $d(x)$). In other words, $t$ is small enough that no image under $d$ of any element in $X
  \setminus \{ x \}$ lies in the circle segment $[d(x), d(x) + t]$. The simplest $f$-representations
  are the one of the form $(x, 0)$, where $x \in X$; this shows that we can embed $X$ into $\fRep(X,
  d)$ by $x \mapsto (x, 0)$.
  
  In Example~\ref{cyclic-group-example}, we have $\fRep(G, d) = G \times [0, 1)$. Moreover, the
  $f$-representations of various elements of the example in Figure~\ref{fig:circle} are shown in
  Figure~\ref{fig:circle-frep}.
  
  We obtain the following result which shows that the set of $f$-representations can be identified
  with the group $(\R/R\Z, +)$.
  
  \begin{proposition}
    \label{onedimfRepBijection}
    The map \[ \hat{d} : \fRep(X, d) \to \R/R\Z, \qquad (x, t) \mapsto d(x) + t \] is a
    bijection. \qed
  \end{proposition}
  
  This allows us to pull the group operation of $\R/R\Z$ back to the set $\fRep(X, d)$, giving an
  operation $+$ on $\fRep(X, d)$ by $(x, t) + (x', t') := \hat{d}^{-1}(\hat{d}(x, t) + \hat{d}(x',
  t'))$. The following remark describes an algorithm which computes the group operation on $\fRep(X,
  d)$ using baby and giant steps.
  
  \begin{remark}
    \label{frepOneDimArithmetic}
    For $(x, t), (x', t') \in \fRep(X, d)$, consider \[ (x'', t'') := (\gs(x, x'), t + t' + (d(x) +
    d(x') - d(\gs(x, x')))) \in X \times \R; \] this ensures that $d(x'') + t'' = \hat{d}(x, t) +
    \hat{d}(x', t')$. In general, $(x'', t'') \not\in \fRep(X, d)$, but $t'' \ge 0$ is not too big;
    more precisely, $t'' < 3 d_{\max}$. The idea of the algorithm for realizing the group operation
    on $\fRep(X, d)$ is to decrease $t''$ using baby steps, while preserving the invariant $d(x'') +
    t'' = \hat{d}(x, t) + \hat{d}(x', t')$, until $(x'', t'') \in \fRep(X, d)$.
    
    For that, note that for $t'' \ge 0$, we have $(x'', t'') \in \fRep(X, d)$ if and only if $t'' <
    d(\bs(x'')) - d(x'')$, i.e.\ if $t''$ is smaller than the distance from $x$ to $\bs(x)$. Hence,
    we iteratively replace $(x'', t'')$ by $(\bs(x''), t'' - (d(\bs(x'')) - d(x'')))$ as long as
    $t'' \ge 0$ is satisfied.

    The smallest non-negative $t''$ yields $(x'', t'') \in \fRep(X, d)$ with $\hat{d}(x'', t'') =
    \hat{d}(x, t) + \hat{d}(x', t')$, and therefore $(x'', t'')$ is the sum of $(x, t)$ and $(x',
    t')$ in $\fRep(X, d)$.
    
    Finally, if we define $d_{\min} := \min\{ d(\bs(x)) - d(x) \mid x \in X \}$, we see that this
    process requires at most $\frac{3 d_{\max}}{d_{\min}}$ baby step computations and one giant step
    computation.
  \end{remark}
  
  The algorithm first uses giant steps to compute a pair~$(x'', t'') \in X \times \R$ with $d(x'') +
  t'' = \hat{d}(x, t) + \hat{d}(x', t')$, where $t''$ is ``small'', and then ``reduces'' $(x'',
  t'')$ to an element of $\fRep(X, d)$. To make this more precise, we need to introduce a
  \emph{reduction map}~$\reduce_{(X, d)} : \R/R\Z \to X$. For a point~$s$ on the circle $\R/R\Z$, we
  want $\reduce_{(X, d)}(s)$ to be the preimage of the element in $d(X)$ ``before'' $s$. More
  precisely, we want $\reduce_{(X, d)}(s)$ to be the unique element in $X$ such that \[ \{
  d(\reduce_{(X,d)}(s)) \} = d(X) \cap [d(\reduce_{(X,d)}(s)), s], \] i.e.\ $d(\reduce_{(X,d)}(s))$
  is the only point in $d(X)$ lying on the circle segment $[d(\reduce_{(X,d)}(s))$, $s]$. Hence,
  $\reduce_{(X,d)}$ assigns to each $s \in \R/R\Z$ some $x \in X$ such that $d(x) \approx s$, and
  satisfies $\reduce_{(X,d)}(d(x)) = x$. The algorithm in Remark~\ref{frepOneDimArithmetic} computes
  $\reduce(d(x'') + t'')$ for $t'' \ge 0$; one can easily adjust it to work for $t'' < 0$ as well.
  
  If one compares the definition of $\reduce_{(X,d)}$ to Figure~\ref{fig:circle-frep}, one quickly
  sees that if $(x, t) \in \fRep(X, d)$ represents $s$, i.e.\ if $d(x) + t = s$, then
  $\reduce_{(X,d)}(s) = x$. Hence, if $\pi : X \times \R \to X$ denotes the projection onto the
  first component, we see that \[ \reduce_{(X,d)}(s) = \pi_1(\hat{d}^{-1}(s)). \] In the context of
  Example~\ref{cyclic-group-example}, where we obtained a one-dimensional infrastructure~$(G, d)$
  from a finite cyclic group~$G = \ggen{g}$, we see that $\reduce(s + R\Z) = g^{\lfloor s \rfloor}$
  for $s \in \R$. This directly follows from the fact that $\fRep(G, d) = G \times [0, 1)$.
  
  Moreover, one can see that the reduction map $\reduce_{(X,d)}$ can be used to define $\fRep(X, d)$
  and giant steps, as 
  \begin{align*}
    \fRep(X, d) ={} & \{ (x, t) \in X \times \R \mid \reduce_{(X,d)}(d(x) + t) = x \} \\ \text{and}
    \qquad\quad \gs(x, y) ={} & \reduce_{(X,d)}(d(x) + d(y)) \qquad \text{for all } x, y \in X.
  \end{align*}
  It is obvious that our choice of $\reduce_{(X,d)}$ is not the only one possible. One could choose
  $\reduce_{(X,d)}$ such that $d(\reduce_{(X,d)}(s))$ is closest to $s$, with a rule to break ties;
  such a reduction map is for example used in \cite{galbraithmirelesmorales-balanced} in the case of
  infrastructures obtained from real quadratic function fields. The advantage of such a reduction
  map is that it reduces the number of baby steps in Remark~\ref{frepOneDimArithmetic} to at most
  $\frac{3 d_{\max}}{2 d_{\min}}$. Using a different reduction map would result in different
  $f$-representations and possibly also different giant steps. We will investigate this relationship
  between reduction maps and $f$-representations in more detail in the next section.
  
  An interesting question is where and how infrastructures occur in practice. The first known
  non-associative instance was the infrastructure of a real quadratic number field, which was
  discovered in 1972 by Shanks. It was originally described in terms of binary quadratic forms, but
  an alternative and more accessible description uses ideals; see, for example,
  \cite{williams-contfract-numbtheor-compus}. We will use the language of ideals since it is
  available in all number fields and function fields. See Section~\ref{reducedideals} on how the
  infrastructure can be realized in detail; for the moment, we want to give a simpler example: the
  infrastructure of a real quadratic number field.
  
  \begin{example}[compare \cite{williams-contfract-numbtheor-compus}]
    \label{realquadratic-example}
    Let $K = \Q(\sqrt{D})$ be a real quadratic number field, where $D > 1$ is a squarefree
    integer. Note that there are two embeddings $K \to \R$, one is the identity, and the other one
    maps $\sqrt{D}$ to $-\sqrt{D}$. Denote the first embedding by $\sigma_1$ and the second one by
    $\sigma_2$; then we have $S = \{ \frakp_1, \frakp_2 \}$ with $\abs{h}_{\frakp_i} =
    \abs{\sigma_i(h)}$ for $h \in K$.
    
    We say that a fractional ideal~$\fraka \in \Id(\calO_K)$ is \emph{reduced} if $1 \in \fraka$,
    and for every $\mu \in \fraka$ satisfying $\abs{\mu}_1 \le 1$ and $\abs{\mu}_2 \le 1$ we have
    $\mu \in \{ -1, 0, 1 \}$. Using the Minkowski embedding $\Phi : K \to K \otimes_\Q \R \cong
    \R^2$, given by $h \mapsto (\sigma_1(h), \sigma_2(h))$, we can visualize $\fraka$ as a
    lattice~$\Phi(\fraka)$ of rank~two in $\R^2$. The condition that $\fraka$ is reduced is
    equivalent to the property that the square $[-1, 1]^2$ contains exactly the three points $(-1,
    -1)$, $(0, 0)$ and $(1, 1)$ of $\Phi(\fraka)$. The unit ideal~$\calO_K$ is always reduced. See
    Figure~\ref{realquaraticfig} for an example.
    
    \begin{figure}[t]
      \newgray{reallylightgray}{0.85}
      \newcommand{\tinyfontsize}{\tiny\darkgray}
      \psset{unit=0.8cm}
      \begin{center}
        \begin{pspicture}(-5,-2)(5,2.25)
          \psline[linecolor=darkgray,linewidth=0.5pt,arrows=->,arrowsize=3pt 3,arrowinset=0](-4,0)(4,0)
          \psline[linecolor=darkgray,linewidth=0.5pt,arrows=->,arrowsize=3pt 3,arrowinset=0](0,-2)(0,2)
          \rput[b](0,2.05){\tiny $\sigma_2$}
          \rput[l](4.05,0){\tiny $\sigma_1$}
          
          \psdot(3.828,-1.828)
          \rput[bl](3.928,-1.828){\tinyfontsize $1 + 2 \sqrt{2}$}
          
          \psframe[linecolor=darkgray,fillstyle=solid,fillcolor=reallylightgray](-1,-1)(1,1)
          \psline[linecolor=darkgray,linewidth=0.5pt,linestyle=dashed](-1,0)(1,0)
          \psline[linecolor=darkgray,linewidth=0.5pt,linestyle=dashed](0,-1)(0,1)
          
          \psdot(1.414,-1.414)
          \psdot(2.414,-0.414)
          \psdot(3.414,0.586)
          \rput[bl](1.514,-1.414){\tinyfontsize $\sqrt{2}$}
          \rput[bl](2.514,-0.414){\tinyfontsize $1 + \sqrt{2}$}
          \rput[bl](3.514,0.586){\tinyfontsize $2 + \sqrt{2}$}
          
          \psdot(-2,-2)
          \psdot(-1,-1)
          \psdot(0,0)
          \psdot(1,1)
          \psdot(2,2)
          \rput[br](-2.1,-2){\tinyfontsize $-2$}
          \rput[br](-1.1,-1){\tinyfontsize $-1$}
          \rput[lb](0.1,0.1){\tinyfontsize $0$}
          \rput[tl](1.1,1){\tinyfontsize $1$}
          \rput[tl](2.1,2){\tinyfontsize $2$}
          
          \psdot(-3.414,-0.568)
          \psdot(-2.414,0.414)
          \psdot(-1.414,1.414)
          \rput[tr](-3.514,-0.568){\tinyfontsize $-2-\sqrt{2}$}
          \rput[tr](-2.514,0.414){\tinyfontsize $-1-\sqrt{2}$}
          \rput[tr](-1.514,1.414){\tinyfontsize $-\sqrt{2}$}
          
          \psdot(-3.828,1.828)
          \rput[tr](-3.928,1.828){\tinyfontsize $-1 - 2 \sqrt{2}$}
        \end{pspicture}
      \end{center}
      \caption{The Minkowski embedding $(\sigma_1, \sigma_2)$ of $\calO_K$ in the real quadratic
      number field $K = \Q(\sqrt{2})$. The grey square is $[-1, 1]^2$.}
      \label{realquaraticfig}
    \end{figure}
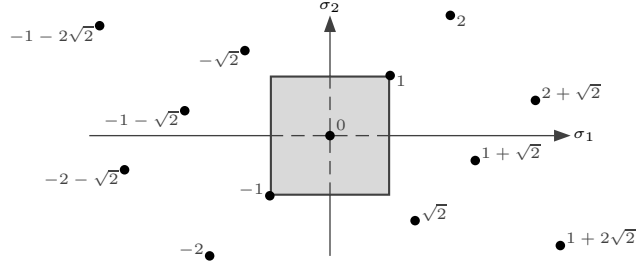
    
    Let $\varepsilon \in \calO_K$ be the fundamental unit with $\varepsilon > 1$. We have $\calO_K^*
    = \{ \pm \varepsilon^n \mid n \in \Z \}$. Set $R := \log \varepsilon$; then $R$ is the regulator
    of $K$. If $\fraka = \frac{1}{\mu} \calO_K$ is a reduced fractional ideal, the elements in $\mu
    \calO_K^*$ are exactly the elements~$\mu'$ such that $\fraka = \frac{1}{\mu'} \calO_K$, whence
    $\{ -\log \abs{\mu'} \mid \fraka = \frac{1}{\mu'} \calO_K \} = -\log \abs{\mu} + R \Z$. Define
    $d(\frac{1}{\mu} \calO_K) := -\log \abs{\mu} + R \Z$; then $d$ is a map from the set~$X$ of
    reduced principal ideals to $\R/R\Z$. One can show that $X$ is finite and that $d$ is
    injective.\footnote{\label{RQFootnote}This is a special case of
    Proposition~\ref{distanceInjectivityProp} with $X = \Red(\calO_K)$ and $d = d^{\calO_K}$, when
    we identify an equivalence class $[\fraka]_\sim$ with $\fraka$, since by
    Corollary~\ref{idealcomparecorr} every class contains exactly one ideal. In this special case,
    at least the injectivity of $d$ is rather obvious, since $h \mapsto \log \abs{h}$ is a group
    homomorphism $(K^*, \cdot) \to (\R, +)$ with kernel $\{ -1, 1 \} = k^* \subseteq \calO_K^*$.}
    Then $(X, d)$ is a one-dimensional infrastructure: in fact, this is the infrastructure used by
    Shanks in \cite{shanks-infra}, translated into the language of ideals.
    
    Computation of baby steps and giant steps is done by continued fraction expansion. Let $\fraka$
    be a principal fractional ideal with $\fraka \cap \Q = \Z$; any reduced ideal satisfies this. We
    can then write $\fraka = \Z \oplus \phi \Z$ with $\phi = (P + \sqrt{D})/Q$, and compute the
    continued fraction expansion of $\phi = \phi_0$. If $\phi_i$ is the $i$-th complete quotient, we
    can write $\phi_i = (P_i + \sqrt{D})/Q_i$ with $P_i, Q_i \in \Z$, and it turns out that
    $\fraka_i := \Z \oplus \phi_i \Z$ is a principal fractional ideal. There exists some $i_0 \in
    \N$, depending on $\fraka$, such that for all $i \ge i_0$, $\fraka_i$ is reduced. In fact, $\{
    \fraka_i \mid i \ge i_0 \}$ is the set of \emph{all} reduced principal ideals~$X$. If $\fraka_i$
    is reduced, $\bs(\fraka_i) = \fraka_{i+1}$. Moreover, if one defines $\reduce(\fraka) =
    \fraka_n$ if $n \ge 0$ is chosen minimal under the condition that $\fraka_n$ is reduced, then
    $\gs(\fraka_i, \fraka_j) = \reduce(\fraka_i \fraka_j)$ is the giant step operation used by
    Shanks in \cite{shanks-infra}.
    
    One can use this to define a reduction map on a dense subset of $\R/R\Z$. For that, note that
    the map $\Psi : \PId(K) \to \R/R\Z$, $\frac{1}{\mu} \calO_K \mapsto -\log \abs{\mu} + R\Z$ is
    injective as argued in Footnote~\ref{RQFootnote}. The set $\Psi(\PId(K))$ is a dense subgroup of
    $\R/R\Z$, and if $s \in \R/R\Z$ lies in the image of $\Psi$, one can define $\reduce(s) :=
    \reduce(\Psi^{-1}(s))$. This was in fact done by Lenstra in \cite{lenstra-infrastructure}, and
    Lenstra called the image of $\Psi$ a ``circular group'' since it is a dense subset of the
    circle~$\R/R\Z$. (Note that Lenstra uses a distance map that is different from the one
    introduced by Shanks.)
  \end{example}
  
  \section{$n$-Dimensional Infrastructures}
  \label{ndiminfra}
  
  In this section, we want to define an abstract $n$-dimensional infrastructure. We want this
  definition to share most properties of one-dimensional infrastructures. Unfortunately, it is not
  as clear as in the one-dimensional case how baby steps, giant steps and $f$-representations can be
  defined. We will see that as outlined in the discussion of the one-dimensional case in the
  previous section, $f$-representations and reduction maps are equivalent, and both yield giant
  steps. As a consequence of the additional freedom gained in the $n$-dimensional case, we are
  forced to include more information on the infrastructure in the definition, namely a reduction
  map.
  
  Note that we will ignore baby steps for the rest of this paper. For $n$-dimensional
  infrastructures obtained from global fields one can define $n + 1$ baby step functions, see for
  example \cite{genvoronoiIuII}, \cite{lee-scheidler-yarrish-cyclic} or
  \cite[Section~3.5]{fontein-diss}. These definitions come from the relation of the infrastructure
  to the set of minima of an ideal, but there is no reason an abstract $n$-dimensional
  infrastructure arises from such a structure. Moreover, such baby steps do not always behave as
  expected, as it might happen that certain minima cannot be reached by baby steps. So far, it is
  unknown whether there is a usable definition of baby steps for abstract $n$-dimensional
  infrastructures when $n > 1$.
  
  We want to make the definition on an $n$-dimensional infrastructure slightly more general by
  allowing to restrict to a suitable subgroup~$\G$ of $\R$. For example, for infrastructures
  obtained from function fields, the natural subgroup to restrict to is $\Z$, since all valuations
  of a function field are discrete. In case of Example~\ref{realquadratic-example}, one could
  restrict to the subgroup~$\{ \log \abs{\mu} \mid \mu \in K^* \}$; then the function $\reduce$ from
  the example will no longer be partially defined, and one essentially obtains the group (though not
  the distance function) of Lenstra \cite{lenstra-infrastructure}.
  
  Throughout this section, fix a suitable non-zero subgroup $\G$ of $\R$. The similarity to
  Section~\ref{odi-intro} is clearer if one assumes $\G = \R$. In the following sections, we will
  restrict to $\G = \Z$ in the function field case and $\G = \R$ in the number field case.
  
  The natural analogue to a circle~$\R/R\Z$ in $n$~dimensions is an $n$-dimensional
  \emph{torus}~$\R^n / \Lambda$, where $\Lambda$ is a lattice of full rank. Since we want to
  restrict to $\G$, we assume that $\Lambda \subseteq \G^n$. Moreover, we abuse terminology by
  calling $\G^n / \Lambda$ a torus, even though we can in general only embed it canonically into the
  torus~$\R^n / \Lambda$. Note that both the circle~$\R/R\Z$ and the torus~$\G^n/\Lambda$ have a
  fixed point~$0$.
  
  A natural generalization of a one-dimensional infrastructure would be a finite set $X \neq
  \emptyset$ together with an injective map $d : X \to \G^n / \Lambda$. Unfortunately, the situation
  is not as simple as in the one-dimensional case. The problem lies in the definition of
  $f$-representations and the giant step function, not to mention the baby step function(s). In the
  one-dimensional case, one has essentially two directions on the circle: one can go clockwise and
  counterclockwise. In fact, our circle $\R/R\Z$ has a distinguished direction corresponding to the
  positive direction on the real line. This allows us to define baby steps as going ``forward'', and
  we can define giant steps, $f$-representations and the reduction map by taking an element of
  $d(X)$ ``before'' a point on the circle.
  
  As soon as $n > 1$, the torus $\G^n/\Lambda$ has infinitely many directions, none of them more
  distinguished than others. This gives many more choices for giant steps, $f$-representations and
  reduction maps, not to mention baby steps. We will be forced to include a particular choice in the
  definition of an $n$-dimensional infrastructure. Before we do that, let us formalize the notions
  of reduction maps and $f$-representations and discuss their relationship.
  
  Let $X$ be a finite set and $d : X \to \G^n / \Lambda$ be a injective map.
  
  \begin{definition}\label{reductionmap-frep-definition}\hfill
    \begin{enuma}
      \item a \emph{reduction map} (for $(X, d)$) is a map~$\reduce : \G^n / \Lambda \to X$ satisfying
      $\reduce(d(x)) = x$ for every $x \in X$;
      \item \emph{$f$-representations} (for $(X, d)$) are a subset $\fRep \subseteq X \times \G^n$
      satisfying $X \times \{ 0 \} \subseteq \fRep$ such that the map \[ \Phi : \fRep \to \G^n /
      \Lambda, \qquad (x, t) \mapsto d(x) + t \] is a bijection.
    \end{enuma}
  \end{definition}
  
  If $(X, d)$ is a one-dimensional infrastructure, then the definition of $\reduce_{(X, d)}$ as in
  Section~\ref{odi-intro} yields a reduction map in the sense of (a), and
  Definition~\ref{onedimfRepDef} and Proposition~\ref{onedimfRepBijection} yield $f$-representations
  in the sense of (b).
  
  Note that the condition $\reduce(d(x)) = x$ for reduction maps ensures that the only fixed points
  under the map $d \circ \reduce : \G^n / \Lambda \to \G^n / \Lambda$ are the elements in $d(X)$;
  i.e.\ these elements can be interpreted as \emph{reduced} elements: other elements of $\G^n /
  \Lambda$ will be mapped to a reduced element when applying $\reduce$, while reduced elements are
  left unchanged under this map.
  
  We begin by outlining the relationship between reduction maps and $f$-re\-pre\-sen\-ta\-tions in
  the sense of Definition~\ref{reductionmap-frep-definition}. This is analogous to the relationship
  in the one-dimensional case in Section~\ref{odi-intro}.
  
  If $\reduce$ is a reduction map, then we obtain a set of $f$-representations by \[ \fRep := \{ (x,
  t) \in X \times \G^n \mid \reduce(d(x) + t) = x \}. \] Here, we choose pairs~$(x, t)$ such that
  $d(x) + t \in \G^n / \Lambda$ will reduce to $x$. If $\reduce$ satisfies $d(\reduce(s)) \approx s$
  for all $s \in \G^n / \Lambda$, then the permissible $t$ values in the $f$-representations will be
  ``small''. Moreover, the condition $\reduce(d(x) + t) = x$ ensures that there is a \emph{unique}
  $f$-representation $(x, t)$ for every $s \in \G^n / \Lambda$.
  
  Conversely, if $\fRep$ is a set of $f$-representations with induced bijection $\Phi : \fRep \to
  \G^n/\Lambda$, then we get a reduction map by \[ \reduce : \G^n / \Lambda \to X, \qquad s \mapsto
  \pi_1(\Phi^{-1}(s)), \] where $\pi_1 : X \times \G^n \to X$ is the projection onto the first
  component. This is the direct generalization of the map $\reduce_{(X, d)}$ in the one-dimensional
  case: given a point $s$ on the torus $\G^n / \Lambda$, we consider the $f$-representation $(x, t)
  = \Phi^{-1}(s)$ representing that point, and return $x = \pi_1(x, t)$.
  
  Therefore, as in the one-dimensional case, the concepts of reduction maps and of
  $f$-representations are equivalent. We can continue as in the one-dimensional case and define
  giant steps using these two notions. If $\reduce : \G^n / \Lambda \to X$ is a reduction map, we
  define \[ \gs(x, y) := \reduce(d(x) + d(y)) \] for all $x, y \in X$; if $\fRep$ are
  $f$-representations with induced bijection $\Phi : \fRep \to \G^n/\Lambda$, we define \[ \gs(x, y)
  := \pi_1(\Phi^{-1}(\Phi(x, 0) + \Phi(y, 0))) \] for all $x, y \in X$. Both definitions yield the
  same giant step operation on $X$.
  
  This discussion gives rise to the following definition of an abstract $n$-dimensional
  infrastructure:
  
  \begin{definition}
    \label{ndiminfradef}
    Let $\Lambda \subseteq \G^n$ be a lattice of full rank.
    \begin{enuma}
      \item An \emph{$n$-dimensional infrastructure} is a triple~$(X, d, \reduce)$, where $X \neq
      \emptyset$ is a non-empty finite set, $d : X \to \G^n / \Lambda$ an injective map and $\reduce
      : \G^n / \Lambda \to X$ a reduction map for $(X, d)$.
      \item If $(X, d, \reduce)$ is an $n$-dimensional infrastructure, then set
      \begin{align*}
        \fRep(X, d, \reduce) :={} & \{ (x, t) \in X \times \G^n \mid \reduce(d(x) + t) = x \} \\
        \text{and} \qquad\qquad\quad \gs(x, x') :={} & \reduce(d(x) + d(y)) \qquad \text{for } x, x'
        \in X.
      \end{align*}
    \end{enuma}
  \end{definition}
  
  Since $R\Z$ is a lattice in $\R^1$ of full rank, we see that a one-dimensional infrastructure~$(X,
  d)$ in the sense of Definition~\ref{odi-definition} is a $1$-dimensional infrastructure~$(X, d,
  \reduce_{(X,d)})$ in the sense of Definition~\ref{ndiminfradef}, whose giant steps and
  $f$-re\-pre\-sen\-ta\-tions coincide. This shows that our new definition is indeed a
  generalization of the notion of a one-dimensional infrastructure as in Section~\ref{odi-intro} or
  \cite{ff-pohlighellman}.
  
  We conclude this section with an example, which shows that $n$-dimensional infrastructures can be
  seen as a generalization of finite abelian groups. Recall that Example~\ref{cyclic-group-example}
  showed how a finite cyclic group can be interpreted as a one-dimensional infrastructure, where the
  distance map was essentially the discrete logarithm map.
  
  \begin{example}
    \label{abelian-group-example}
    Assume that $\Z \subseteq \G$. Let $G = \ggen{g_1, \dots, g_n}$ be a finite abelian group, and
    let \[ \Lambda := \biggl\{ (e_1, \dots, e_n) \in \Z^n \;\biggm|\; \prod_{i=1}^n g_i^{e_i} = 1
    \biggr\} \] be the relation lattice of $g_1, \dots, g_n$; this is the kernel of the epimorphism
    $\Z^n \to G$, $(e_1, \dots, e_n) \mapsto \prod_{i=1}^n g_i^{e_i}$, whence $G \cong \Z^n /
    \Lambda$.
    
    Concatenating the inverse of this isomorphism with the inclusion $\Z^n / \Lambda \subseteq \G^n
    / \Lambda$, we obtain an injective map $d : G \to \G^n / \Lambda$. This map is the
    \emph{generalized discrete logarithm} map with base $g := (g_1, \dots, g_n)$, i.e.\ it
    satisfies\footnote{For $g = (g_1, \dots, g_n) \in G^n$ and $v = (v_1, \dots, v_n) \in \Z^n$,
    define $g^v := \prod_{i=1}^n g_i^{v_i}$.}  $g^{d(h)} = h$ for every $h \in G$.
    
    It is easy to see that $\fRep := G \times (\G \cap [0, 1))^n$ is a set of $f$-representations
    for $(G, d)$; the corresponding reduction map maps~$s \in \G^n/\Lambda$ to $\reduce(s) :=
    \prod_{i=1}^n g_i^{\lfloor e_i \rfloor}$, if $s = (e_i)_i + \Lambda$. Therefore, $(G, d,
    \reduce)$ is an $n$-dimensional infrastructure. The induced giant step map is given by \[ \gs(h,
    h') = \reduce(d(h) + d(h')) = h h' \] for $h, h' \in G$, since $d(h) + d(h') = (e_1, \dots, e_n)
    + \Lambda$ with $e_i \in \Z$ and $g_1^{e_1} \cdots g_n^{e_n} = h h'$.
    
    This shows that giant steps generalize the group operation in this case as well. In particular,
    in this case, the giant step operation is associative, as opposed to general $n$-dimensional
    infrastructures.
  \end{example}
  
  \section{Reduced Ideals}
  \label{reducedideals}
  
  Now that we have obtained a definition of an abstract $n$-dimensional infrastructure, we want to
  construct such an infrastructure from a global field~$K$. The aim of this section is to construct
  the lattice $\Lambda \subset \G^n$, the finite set $X$ as well as the injective map~$d : X \to
  \G^n / \Lambda$. In the next section, we will add a reduction map $\reduce$ for $(X, d)$ such that
  $(X, d, \reduce)$ is an $n$-dimensional infrastructure.
  
  For the rest of the paper, let $\G$ denote $\Z$ if $K$ is a function field and $\R$ if $K$ is a
  number field.
  
  In order to construct the underlying set $X$, we require the notion of a reduced (fractional)
  ideal\footnote{Recall that we always mean ``non-zero fractional ideal'' when we write ``ideal'',
  if not explicitly said otherwise.} in analogy to Example~\ref{realquadratic-example}. In case $K$
  is a function field, reduced ideals correspond to certain reduced divisors in the sense of
  \cite{hessRR}.
  
  The notion of a reduced ideal is rather geometric. To describe it, we define the notion of a
  \emph{box}, which is the set of elements of an ideal (interpreted as a lattice) in a bounded
  area. An ideal will be reduced if a certain box contains elements only at very specific
  positions. Write $S = \{ \frakp_1, \dots, \frakp_{n+1} \}$, where $n = \abs{S} - 1$; then recall
  that the absolute value of an element~$h \in K$ with respect to a place~$\frakp$ is defined as
  $\abs{h}_\frakp = q^{-\nu_\frakp(h) \deg \frakp}$, where $q > 1$ is a constant.
  
  For $t_1, \dots, t_{n+1} \in \G$ and an ideal~$\fraka \in \Id(\calO_K)$, we define \[ B(\fraka,
  (t_1, \dots, t_{n+1})) := \{ h \in \fraka \mid \forall i \in \{ 1, \dots, n+1 \} :
  \abs{h}_{\frakp_i} \le q^{t_i \deg \frakp_i} \}. \] The motivation of this definition comes from
  the number field case; in that scenario, $\fraka$ is a lattice of full rank under the Minkowski
  embedding $K \hookrightarrow K \otimes_\Q \R \cong \R^d$, where $d = [K : \Q]$. The box $B(\fraka,
  (t_1, \dots, t_{n+1}))$ is the set of lattice points lying in the symmetric compact convex set
  described by $(t_1, \dots, t_{n+1})$. If $K$ is totally real, this convex set is a hyperrectangle
  (box) with side lengths $2 e^{t_1}, \dots, 2 e^{t_{n+1}}$, and if $K$ is totally imaginary, this
  convex set is the direct product of $n + 1$ closed discs of radii~$e^{t_1}, \dots,
  e^{t_{n+1}}$. If $K$ is neither totally real nor totally complex, the convex set features both
  properties; for example, if $K$ has one real embedding corresponding to $\frakp_1$ and two complex
  conjugate embeddings corresponding to $\frakp_2$, the convex set is a cylinder of length~$2
  e^{t_1}$ and radius~$e^{t_2}$. Figure~\ref{realquaraticfig} on page~\pageref{realquaraticfig}
  displays the box with parameters~$t_1 = t_2 = 0$ in the real quadratic number field~$K =
  \Q(\sqrt{2})$ as the grey square in the center.
  
  If $\mu \in K^*$, we define the abbreviation \[ B(\fraka, \mu) := B(\fraka, (-\nu_{\frakp_1}(\mu),
  \dots, -\nu_{\frakp_{n+1}}(\mu))). \] This is the smallest box which would containing $\mu$ if
  $\mu \in \fraka$. With this, we are able to define reduced ideals:
  \begin{definition}\hfill
    \label{minimumreduceddef}
    \begin{enuma}
      \item An element $\mu \in \fraka \setminus \{ 0 \}$ is said to be a \emph{minimum} of $\fraka$
        if for every $h \in B(\fraka, \mu)$ we either have $h = 0$ or $\abs{h}_\frakp =
        \abs{\mu}_\frakp$ for all $\frakp \in S$. Denote the set of all minima of $\fraka$ by
        $\calE(\fraka)$.
      \item An  ideal $\fraka$ is said to be \emph{reduced} if $1 \in \fraka$ is a minimum
        of $\fraka$.
    \end{enuma}
  \end{definition}
  The notation of $\calE(\fraka)$ for the set of minima goes back to Y.~Hellegouarch and
  R.~Paysant-Le~Roux \cite{hellegouarchpaysant}.
  
  The property that $\mu$ is a minimum of $\fraka$ means simply that the box $B(\fraka, \mu)$ is
  empty, up to a few elements which always need to belong to $B(\fraka, \mu)$: $0$ is always
  contained in $B(\fraka, \mu)$, as well as $\mu$ and $\varepsilon \mu$ for all $\varepsilon \in
  k^*$, since all absolute values of elements in $k^*$ are $1$. Hence, we ask that all elements in
  $B(\fraka, \mu)$ are either $0$ or have the same infinite absolute values as $\mu$. For example,
  in Figure~\ref{realquaraticfig}, $1$ and $1 + \sqrt{2}$ are minima of $\calO_K$, while $\sqrt{2}$
  is not, since $1 \in B(\calO_K, \sqrt{2}) \setminus \{ 0 \}$ has different absolute values than
  $\sqrt{2}$.
  
  Under certain circumstances, there can be elements in $\fraka$ with the same infinite absolute
  values as $\mu$ other than $\varepsilon \mu$, $\varepsilon \in k^*$, and these elements thus
  belong to $B(\fraka, \mu)$ as well. These elements are the reason why the aforementioned
  equivalence relation is needed: if $\mu \in \fraka$ is such an element, then $\frac{1}{\mu}
  \fraka$ is a reduced ideal different from $\fraka$ which will be mapped to the same element by our
  distance map. For that reason, we have to identify any two such ideals if such elements can exist.

  The following proposition shows that in many important situations, such elements cannot
  occur. This includes in particular the case when an infinite place of degree~one exists. For
  number fields~$K$, this is always the case unless $K$ is totally imaginary, and for function
  fields one can always move to a constant field extension by a splitting field for one of the
  infinite places. Many treatments of the infrastructure and of arithmetic in function fields
  require such a place of degree~one, sometimes explicitly as for He\ss' arithmetic \cite{hessRR}
  and sometimes implicitly by restricting to certain classes of fields; for example, every real
  quadratic field has exactly two infinite places of degree~one, and a cubic number field always has
  a real embedding.
  
  \begin{proposition}
    \label{reducedidealdeg1prop}
    Assume that $\deg \frakp = 1$ for some $\frakp \in S$. Let $\frakb$ be a reduced ideal. Then
    $B(\frakb, (0, \dots, 0)) = k$.
  \end{proposition}
  
  Before we proceed with the proof, we need to introduce a right inverse $\divisor : \Id(\calO_K)
  \to \Div(K)$ to the natural map $\Div(K) \to \Id(\calO_K)$ described in
  Section~\ref{fnf-intro}. For a fractional ideal~$\frakb = \prod_{\frakp \not\in S} (\frakm_\frakp
  \cap \calO_K)^{n_\frakp}$, define $\divisor(\frakb) := -\sum_{\frakp \not\in S} n_\frakp
  \frakp$. This allows us to relate boxes to \emph{Riemann-Roch spaces}: we have \[ B(\fraka, (t_1,
  \dots, t_{n+1})) = L\biggl( \divisor(\fraka) + \sum_{i=1}^{n+1} t_i \frakp_i \biggr); \] here,
  $L(D) := \{ f \in K^* \mid (f) \ge -D \} \cup \{ 0 \}$ for $D \in \Div(K)$ is the Riemann-Roch
  space of $D$.
  
  \begin{proof}[Proof of Proposition~\ref{reducedidealdeg1prop}.]
    If $K$ is a number field, then $\deg \frakp = 1$ means that $\frakp$ corresponds to a real
    embedding; hence, $\abs{h}_\frakp = \abs{h'}_\frakp$ for $h, h' \in K$ if and only if $h = \pm
    h'$. Thus, if $\frakb$ is reduced, $B(\frakb, (0, \dots, 0)) = k = \{ -1, 0, 1 \}$.
    
    If $K$ is a function field, then $B(\frakb, (0, \dots, 0)) = L(\divisor(\frakb)) \supseteq k$,
    and $L(\divisor(\frakb) - \frakp) = 0$. But by \cite[Lemma~I.4.8]{stichtenoth},
    \begin{align*}
      0 = \dim_k L(\divisor(\frakb) - \frakp) & {}\le \dim_k L(\divisor(\frakb)) \\ & {}\le \dim_k
      L(\divisor(\frakb) - \frakp) + \deg \frakp = 1,
    \end{align*}
    whence $B(\frakb, (0, \dots, 0)) = k$.
  \end{proof}
  
  For the rest of the section, we fix an ideal~$\fraka \in \Id(\calO_K)$. There is a close
  relationship between the set of minima of an ideal and the set of reduced ideals in the ideal
  class of that ideal. First note that the unit group $\calO_K^*$ of $\calO_K$ operates on
  $\calE(\fraka)$ by multiplication: if $\mu \in \calE(\fraka)$ and $\varepsilon \in \calO_K^*$,
  then $\varepsilon \mu \in \calE(\fraka)$. This shows that the map \[ \calE(\fraka) / \calO_K^* \to
  \Id(\calO_K), \qquad \mu \calO_K^* \mapsto \tfrac{1}{\mu} \fraka \] is well-defined and
  injective. The image of this map is exactly the set of reduced ideals in the ideal class of
  $\fraka$. Denote this set by $\Red(\fraka)$. This set, modulo the aforementioned equivalence
  relation, will be our set $X$.
  
  Note that in case $K$ is a function field with $\deg \frakp_i = 1$ for some $i$, the reduced
  ideals $\frakb \in \Red(K)$, where \[ \Red(K) := \bigcup_{\fraka \in \Id(\calO)} \Red(\fraka), \]
  correspond exactly to the divisors~$D \in \Div(K)$ that are reduced with respect to $\frakp_i$ in
  the sense of He\ss\ \cite{hessRR} and satisfy $\nu_{\frakp_j}(D) = 0$ for $j \in \{ 1, \dots, n +
  1 \}$. This is due to the relationship between boxes and Riemann-Roch spaces sketched above, and
  the correspondence is given by $\frakb \mapsto \divisor(\frakb)$.
  
  Next, we want to construct the distance map~$d$. In the process of constructing $d$, we obtain the
  lattice $\Lambda$ and derive the equivalence relation needed to define $X$.  We begin with the
  map \[ \Psi : K^* \to \G^n, \qquad h \mapsto (-\nu_{\frakp_1}(h), \dots, -\nu_{\frakp_n}(h)), \]
  which maps $(K^*, \cdot)$ homomorphically into $(\G^n, +)$. This map plays a crucial role in
  constructing the distance map. The image of $\calO_K^*$ under $\Psi$ is a lattice in $\G^n
  \subseteq \R^n$; it is called the \emph{unit lattice} of $\calO_K$ and is denoted by
  $\Lambda$. For number fields $K$, $\Lambda$ always has full rank; this is a consequence of
  Dirichlet's Unit Theorem. For function fields, $\Lambda$ has full rank if and only if $T$ is
  finite. In case $\Lambda$ has full rank, we have \[ \det \Lambda = \frac{R}{\prod_{i=1}^n \deg
  \frakp_i}. \] In case $\calO_K^*$ has full rank, $\G^n / \Lambda \subseteq \R^n / \Lambda$ is an
  $n$-dimensional torus that will be the codomain of our distance map~$d$.
  
  Note that $\frac{1}{h} \fraka = \frac{1}{h'} \fraka$ if and only if $h' h^{-1} \in \calO_K^*$, and
  this implies $\Psi(h) - \Psi(h') \in \Lambda$. Therefore, the map $\frac{1}{h} \fraka \mapsto
  \Psi(h) + \Lambda$ is well-defined. Ideally, this map will represent our distance
  map. Unfortunately, it is in general not injective on $\Red(\fraka)$, whence we need to identify
  elements in $\Red(\fraka)$ which are mapped onto the same element of $\G^n / \Lambda$ under the
  map $\frac{1}{h} \fraka \mapsto \Psi(h) + \Lambda$. We will define an equivalence relation~$\sim$,
  study it in more detail, and then show in Proposition~\ref{distanceInjectivityProp} that it indeed
  makes this map injective.
  
  If $\frakb, \frakb'$ are ideals in the ideal class of $\fraka$ such that $\frakb = h \frakb'$ with
  $\abs{h}_\frakp = 1$ for all $\frakp \in S$, then $\frakb$ and $\frakb'$ are mapped to the same
  element of $\G^n / \Lambda$. \label{equivalenceRelationDefinition}Hence, we can define the
  equivalence relation $\sim$ on $\Id(\calO_K)$ by \[ \frakb \sim \frakb' :\Longleftrightarrow
  \exists h \in K^* : \frakb = h \frakb' \wedge \forall \frakp \in S : \abs{h}_\frakp = 1. \] We
  thus see that the map $\Red(\fraka)/_\sim \to \G^n/\Lambda$ via $[\frac{1}{\mu} \fraka]_\sim
  \mapsto \Psi(h) + \Lambda$ is well-defined, but we are left to show that it is injective. Note
  that the above equivalence relation~$\sim$ is \emph{not} the equivalence relation on ideals used
  to define the ideal class group~$\Pic(\calO_K)$: we impose the additional condition that
  $\abs{h}_\frakp = 1$ for all $\frakp \in S$.
  
  We can deduce from Proposition~\ref{reducedidealdeg1prop} that this equivalence relation~$\sim$ is
  trivial in case an infinite place of degree~one exists:
  
  \begin{corollary}
    \label{idealcomparecorr}
    Assume that $\deg \frakp = 1$ for some $\frakp \in S$. Let $\frakb$ and $\frakb'$ be two reduced
    ideals. Then $\frakb \sim \frakb'$ if and only if $\frakb = \frakb'$.
  \end{corollary}
  
  \begin{proof}
    If $\frakb = h \frakb'$ with $\abs{h}_\frakp = 1$ for all $\frakp \in S$, we get $h \in
    B(\frakb, (0, \dots, 0))$ and thus, by Proposition~\ref{reducedidealdeg1prop}, $h \in k^*
    \subseteq \calO_K^*$.
  \end{proof}
  
  In the general case, testing $\sim$ is more complicated. The following proposition shows how this
  can be done:
  
  \begin{proposition}
    \label{idealcompareprop}
    Let $\frakb$ and $\frakb'$ be two reduced ideals. Then $\frakb \sim \frakb'$ if and only if
    $B(\frakb (\frakb')^{-1}, (0, \dots, 0)) \neq \{ 0 \}$ and $\deg \divisor(\frakb) = \deg
    \divisor(\frakb')$.
  \end{proposition}
  
  Note that in case $K$ is a number field, $\deg \divisor(\frakb) = -\log \Norm_{K/\Q}(\frakb)$, and
  in case $K$ is a function field, $\deg \divisor(\frakb) = -\deg \Norm_{K/k(x)}(\frakb)$.
  
  \begin{proof}[Proof of Proposition~\ref{idealcompareprop}.]
    If $\frakb \sim \frakb'$, there exists $h \in K^*$ with $h \frakb' = \frakb$ such that
    $\abs{h}_\frakp = 1$ for all $\frakp \in S$. Hence, $\frakb (\frakb')^{-1} = h \calO_K$ and $h
    \in B(\frakb (\frakb')^{-1}, (0, \dots, 0))$. Moreover, $\divisor(\frakb (\frakb')^{-1}) =
    (h^{-1})$ is principal, i.e.\ of degree~zero, whence $\deg \divisor(\frakb) = \deg
    \divisor(\frakb')$.
    
    Conversely, we see that $\divisor(\frakb (\frakb')^{-1})$ must be principal as $\deg
    \divisor(\frakb (\frakb')^{-1}) = 0$. Hence, there exists $h \in K^*$ with $\divisor(\frakb
    (\frakb')^{-1}) = (h^{-1})$. But then $\frakb (\frakb')^{-1} = h \calO_K$ and $\nu_\frakp(h) = 0$
    for all $\frakp \in S$, i.e.\ $\frakb \sim \frakb'$.
  \end{proof}
  
  We have now obtained a well-defined map $[\frac{1}{h} \fraka]_\sim \mapsto \Psi(h) + \Lambda$, and
  we are able to test whether $\frakb \sim \frakb'$ for ideals~$\frakb, \frakb' \in \Red(K)$. The
  next statement shows that this map is injective when restricted to the non-empty set
  $\Red(\fraka)/_\sim$, and that $\Red(\fraka)/_\sim$ is finite for all global fields and some
  non-global function fields.
  
  \begin{proposition}
    \label{distanceInjectivityProp}
    The map \[ d^\fraka : \Red(\fraka) /_\sim \to \G^n / \Lambda, \qquad [\tfrac{1}{\mu}
      \fraka]_\sim \to \Psi(\mu) + \Lambda \] is injective. In case $K$ is a number field, or $K$ is
    a function field and $T$ is finite, the set $\Red(\fraka) /_\sim$ is finite. In any case, it is
    non-empty. In case $K$ is a global field, $\Red(\fraka)$ itself is finite as well.
  \end{proposition}
  
  Recall that when $K$ is a function field with finite constant field, then $T$ is always
  finite. Note that the injectivity of $d^\fraka$ for number fields is also shown by Schoof in
  \cite[Lemma~9.2~(ii)]{schoofArakelov}. The finiteness result for global fields is well known, see
  for example \cite[Theorems~3 and 4]{hellegouarchpaysant}.
  
  \begin{proof}[Proof of Proposition~\ref{distanceInjectivityProp}.]
    For injectivity, assume that $d^\fraka(\frac{1}{\mu} \fraka) = d^\fraka(\frac{1}{\mu'} \fraka)$;
    this means that $\Psi(\mu) - \Psi(\mu') \in \Lambda$. If we choose $\varepsilon \in \calO_K^*$
    with $\Psi(\varepsilon) = \Psi(\mu) - \Psi(\mu')$, we obtain $\abs{\mu}_\frakp =
    \abs{\varepsilon \mu'}_\frakp$ for every $\frakp \in S$. Therefore, $h := \frac{\mu}{\varepsilon
    \mu'}$ satisfies $\abs{h}_\frakp = 1$ for all $\frakp \in S$, and $h \cdot \frac{1}{\mu} \fraka
    = \frac{1}{\mu'} \varepsilon^{-1} \fraka = \frac{1}{\mu'} \fraka$, whence $\frac{1}{\mu} \fraka
    \sim \frac{1}{\mu'} \fraka$.
    
    To see that $\Red(\fraka)$ is non-empty, it suffices to show that $\fraka$ has at least one
    minimum. This can be done directly using Riemann's Inequality or Minkowski's Lattice Point
    Theorem, or one can use tools as the Reduction Lemma~\ref{reduction-lemma}, applied to $(\fraka,
    (0, \dots, 0))$. It returns a tuple whose first component is the equivalence class of an element
    in $\Red(\fraka)$.
    
    In case $K$ is a function field and $T$ is finite, the finiteness of $\Red(\fraka) / _\sim$
    follows from the fact that $\G^n / \Lambda$ is finite, since $T$ is isomorphic to a subgroup of
    $\G^n / \Lambda$ of finite index. If moreover $k$ is finite, note that the equivalence class
    $[\frakb]_\sim$ of $\frakb$ is finite for every~$\frakb$ since $\frakb = f \frakb'$ with
    $\abs{f}_\frakp = 1$ for all $\frakp \in S$ implies $f \in B(\frakb, (0, \dots, 0))$, which is a
    finite $k$-vector space and thus also a finite set. Therefore, $\Red(\frakb)$ is the union of
    finitely many finite sets.
    
    Finally, in case $K$ is a number field, Remark~\ref{frepremarks}~(b) and
    Proposition~\ref{frepdegreeinbothcases:body} show that if $\frakb$ is a reduced ideal, then
    $\frakb^{-1}$ is an integral ideal with bounded norm. As there are only finitely many of these,
    $\Red(\fraka)$ itself is finite.
  \end{proof}
  
  Assume that $K$ is a global field. If we define $X^\fraka := \Red(\fraka)/_\sim$, we obtain the
  first ingredients of an $n$-dimensional infrastructure: a finite set $X^\fraka$, a lattice
  $\Lambda \subseteq \G^n$ of full rank, and an injective map $d^\fraka : X^\fraka \to \G^n /
  \Lambda$.
  
  The map $d^\fraka$ takes the equivalence class of a reduced ideal~$\frakb$ in the ideal class of
  $\fraka$, say $\frakb = \frac{1}{\mu} \fraka$, and maps it to its ``distance'' $\Psi(\mu)$, which
  is well-defined up to elements of $\Lambda$. Chosing the logarithmic absolute value vector of the
  relative generator~$\mu$ as the distance generalizes Shanks' original definition of distance in
  infrastructures \cite{shanks-infra}, and is used in most treatments of the infrastructure, for
  example in the works of Buchmann and Williams. A notable difference is Lenstra's distance function
  \cite{lenstra-infrastructure}. In the case of function fields, this is also the common measure
  used to define distances, at least in the case of unit rank~one
  \cite{paulus-rueck,scheidler-infrastructurepurelycubic,ericphd}.
  
  In this section, we obtained for every~$\fraka \in \Id(\calO)$ a finite set~$X^\fraka$, a lattice
  of full rank~$\Lambda \subseteq \G^n$, as well as an injective map $d^\fraka : X^\fraka \to \G^n /
  \Lambda$. Here, we needed to assume that $K$ is a global field to ensure that $X^\fraka$ is finite
  and $\Lambda$ is of full rank, though this can also be true for certain function fields with
  infinite constant fields. In fact, for arbitrary function fields, $X^\fraka$ is finite if, and
  only if, $\Lambda$ is of full rank. The ingredient this is still missing in order to obtain an
  $n$-dimensional infrastructure in the sense of Definition~\ref{ndiminfradef}, namely, a reduction
  map, will be defined in the next section.
  
  \section{$f$-Representations in Global Fields}
  \label{geninfra}
  
  In this section we introduce $f$-representations~$\fRep(\fraka)$ for $(\Red(\fraka)/_\sim,
  d^\fraka)$. Using the equivalence of $f$-representations and reduction maps discussed in
  Section~\ref{ndiminfra}, this yields a reduction map~$\reduce^\fraka : \G^n / \Lambda \to X^\fraka
  = \Red(\fraka)/_\sim$, so that $(X^\fraka, d^\fraka, \reduce^\fraka)$ is an $n$-dimensional
  infrastructure in the sense of Definition~\ref{ndiminfradef}.
  
  Before we define $f$-representations for arbitrary number fields and function fields, we want to
  consider a special case, namely~$\deg \frakp_{n+1} = 1$. In this case, the definition of an
  $f$-representation can be drastically simplified and stated with a lot less technical
  involvement. We distinguish the simpler scenario from the general case by appending an asterisk to
  the $f$ in $f$-representations. By Corollary~\ref{idealcomparecorr}, we can
  replace~$\Red(\fraka)/_\sim$ by $\Red(\fraka)$ itself, as every equivalence class $[\frakb]_\sim$
  contains exactly one reduced ideal. Recall that in this case, an ideal~$\frakb \in \Id(\calO_K)$
  is reduced if $B(\frakb, (0, \dots, 0)) = k$ by Proposition~\ref{reducedidealdeg1prop}. An
  $f$-representation should be a reduced ideal~$\frakb$ together with numbers~$t_1, \dots, t_n \ge
  0$ which determine how far the box $B(\frakb, (0, \dots, 0))$ can be enlarged in the directions of
  $\frakp_1, \dots, \frakp_n$ without containing anything but $k$. More precisely:
  
  \begin{definition}
    An \emph{$f\ast$-representation} is a tuple~$(\frakb, (t_1, \dots, t_n)) \in \Red(\fraka) \times
    \G^n$ such that $B(\frakb, (t_1, \dots, t_n, 0)) = k$. Denote the set of all
    $f\ast$-representations in $\Red(\fraka) \times \G^n$ by $\fRepast(\fraka)$.
  \end{definition}
  
  We say that $(\frakb, t) \in \fRepast(\fraka)$ \emph{represents} $d^\fraka([\frakb]_\sim) + t \in
  \G^n/\Lambda$.
  
  \begin{remark}\hfill
    \label{frepremarksast}
    \begin{enuma}
      \item If $\frakb \in \Red(\fraka)$, then always $(\frakb, (0, \dots, 0)) \in
      \fRepast(\fraka)$.
      \item If $\frakb = \frac{1}{\mu} \fraka$ for some $\mu \in K^*$, and if $t_1, \dots, t_n \in
      \G$ are elements such that $B(\frakb, (t_1, \dots, t_n, 0)) = k$, then $(\frakb, (0, \dots,
      0)) \in \fRepast(\fraka)$. In particular, $\frakb \in \Red(\fraka)$. This shows that the
      assumption that $\frakb \in \Red(\fraka)$ in the definition is not actually needed.
    \end{enuma}
  \end{remark}
  
  Now we drop the assumption that $\deg \frakp_{n+1} = 1$. We have to introduce certain
  technicalities to ensure that $f$-representations are well-defined. First, as in the case of
  reduced ideals, we will only have $k \subseteq B(\frakb, (t_1, \dots, t_n, 0))$. To ensure that
  this set does not contain too many elements, we need to introduce a technical tool, namely a total
  preorder\footnote{A total preorder~$\le$ on a set $X$ is a binary relation which is reflexive and
  transitive such that for every $x, y \in X$, we have $x \le y$ or $y \le x$.} on $K^*$. For $h, h'
  \in K^*$, define \[ h \le h' :\Longleftrightarrow (\abs{h}_{\frakp_{n+1}}, \abs{h}_{\frakp_1},
  \dots, \abs{h}_{\frakp_n}) \le_{\ell ex} (\abs{h'}_{\frakp_{n+1}}, \abs{h'}_{\frakp_1}, \dots,
  \abs{h'}_{\frakp_n}), \] where $\le_{\ell ex}$ is the usual lexicographical order on
  $\R^{n+1}$. Using this notion, we define $f$-representations as follows:
  
  \begin{definition}\label{fRepDefForGlobalField}
    An \emph{$f$-representation} is a tuple $([\frakb]_\sim, (t_1, \dots, t_n)) \in
    \Red(\fraka)/_\sim \times \G^n$ such that $1 \in B(\frakb, (t_1, \dots, t_n, 0)) \setminus \{ 0
    \}$ is a smallest element with respect to $\le$. Denote the set of all $f$-representations in
    $\Red(\fraka)/_\sim \times \G^n$ by $\fRep(\fraka)$.
  \end{definition}
  
  As above, we say that $([\frakb]_\sim, t) \in \fRep(\fraka)$ \emph{represents} $d^\fraka([\frakb]_\sim) +
  t \in \G^n/\Lambda$.
  
  The condition that $1$ is a smallest element with respect to $\le$ ensures that all elements $h
  \in B(\frakb, (t_1, \dots, t_n, 0)) \setminus \{ 0 \}$ satisfy $\abs{h}_{\frakp_{n+1}} =
  1$. Moreover, it ensures that $\frakb$ is reduced, since any element in $B(\frakb, (0, \dots, 0))
  \setminus \{ 0 \}$ whose absolute values are not equal to 1 would be strictly less than $1$ with
  respect to this order.
  
  In fact, the choice of $\le$ is somewhat arbitrary. One could replace $\le$ with any other
  preorder on $K^*$ such that:
  \begin{enuma}
    \item if $h, h' \in K^*$ satisfy $\abs{h}_{\frakp_{n+1}} < \abs{h'}_{\frakp_{n+1}}$, then $h <
    h'$;
    \item if $h, h', h'' \in K^*$ satisfy $h' \le h''$, then $h h' \le h h''$;
    \item for every ideal~$\frakb$ and any $t_1, \dots, t_{n+1} \in \G$, the set $B(\frakb, (t_1,
    \dots, t_{n+1})) \setminus \{ 0 \}$ has a smallest element with respect to $\le$ if it is
    non-empty, and this element happens to be a minimum of $\frakb$ in the sense of
    Definition~\ref{minimumreduceddef};
    \item if $h \le h'$ and $h' \le h$ for $h, h' \in K^*$, we have $\abs{h}_\frakp =
    \abs{h'}_\frakp$ for every~$\frakp \in S$.
  \end{enuma}
  The choice of $\le$ as the lexicographical order on vectors of absolute values is a convenient
  choice satisfying these conditions, in particular since it is well-suited for computations.
  
  In case $\deg \frakp_{n+1} = 1$, $f$-representations and the simpler $f\ast$-representations
  coincide; this is shown in Proposition~\ref{deg1freps}. Before we establish this, we state a few
  more remarks.
  
  \begin{remark}\hfill
    \label{frepremarks}
    \begin{enuma}
      \item The definition of an $f$-representation depends only on the equivalence class
      $[\frakb]_\sim$ of $\frakb$: if $\frakb, \frakb'$ are two reduced ideals with $\frakb \sim
      \frakb'$, an element with given absolute values in $B(\frakb, (t_1, \dots, t_n, 0))$ exists if
      and only if an element with the same absolute values exists in $B(\frakb', (t_1, \dots, t_n,
      0))$. Therefore, $\fRep(\fraka)$ is well-defined.
      \item If $(\frakb, (t_1, \dots, t_n)) \in \fRepast(\fraka)$, then $([\frakb]_\sim, (t_1,
      \dots, t_n)) \in \fRep(\fraka)$.
      \item If $\frakb \in \Red(\fraka)$, then always $([\frakb]_\sim, (0, \dots, 0)) \in
      \fRep(\fraka)$. That is, the reduced ideals in the ideal class of $\fraka$, modulo the
      equivalence relation~$\sim$, can be embedded into $\fRep(\fraka)$.
      \item If $\frakb = \frac{1}{\mu} \fraka$ for some $\mu \in K^*$, and if $t_1, \dots, t_n \in
      \G$ are elements such that $1 \in B(\frakb, (t_1, \dots, t_n, 0)) \setminus \{ 0 \}$ is a
      smallest element with respect to $\le$, then $([\frakb]_\sim, (0, \dots, 0)) \in
      \fRep(\fraka)$. In particular, $\frakb \in \Red(\fraka)$. This shows that the assumption that
      $[\frakb]_\sim \in \Red(\fraka)/_\sim$ in the definition is not actually needed.
    \end{enuma}
  \end{remark}
  
  As mentioned before, in case of $\deg \frakp_{n+1} = 1$, $f$-representations are equivalent to the
  simpler $f\ast$-representations introduced first:
  
  \begin{proposition}
    \label{deg1freps}
    The map \[ \fRepast(\fraka) \to \fRep(\fraka), \qquad (\frakb, (t_1, \dots, t_n)) \mapsto
    ([\frakb]_\sim, (t_1, \dots, t_n)) \] is always an injection. If further $\deg \frakp_{n+1} =
    1$, it is a bijection.
  \end{proposition}
  
  \begin{proof}
    The map is well-defined by Remark~\ref{frepremarks}~(b). To see that it is injective, note that
    $k \subseteq B(\frakb, (0, \dots, 0)) \subseteq B(\frakb, (t_1, \dots, t_n, 0)) = k$ for
    $(\frakb, (t_1, \dots$, $t_n)) \in \fRepast(\fraka)$, whence $\frakb \sim \frakb'$ for $\frakb'
    \in \Red(\fraka)$ implies $\frakb = \frakb'$. Therefore, $[\frakb]_\sim$ contains exactly one
    element, whence $([\frakb]_\sim, (t_1, \dots, t_n))$ has exactly one preimage.
    
    To see that the map is surjective in case $\deg \frakp_{n+1} = 1$, let $([\frakb]_\sim, (t_1,
    \dots, t_n)) \in \fRep(\fraka)$. Note that $\abs{h}_{\frakp_{n+1}} = 1$ for all $h \in B(\frakb,
    (t_1, \dots, t_n, 0)) \setminus \{ 0 \}$. We can proceed in a very similar manner as in the
    proof of Proposition~\ref{reducedidealdeg1prop}. In case $K$ is a number field, this shows that
    $B(\frakb, (t_1, \dots, t_n, 0)) = \{ -1, 0, 1 \} = k$.
    
    In case $K$ is a function field, $B(\frakb, (t_1, \dots, t_n, 0)) = L(D)$ with $D :=
    \divisor(\frakb) + \sum_{i=1}^n t_i \frakp_i$, and we know that $L(D - \frakp_{n+1}) = \{ 0
    \}$. As in the proof of Proposition~\ref{reducedidealdeg1prop}, we must have $\dim_k L(D) = 1$,
    whence $1 \in L(D)$ implies $L(D) = k$.
    
    So in both cases, $B(\frakb, (t_1, \dots, t_n, 0)) = k$, whence $(\frakb, (t_1, \dots, t_n)) \in
    \fRepast(\fraka)$ is a preimage of $([\frakb]_\sim, (t_1, \dots, t_n)) \in \fRep(\fraka)$.
  \end{proof}
  
  Before we show that $\fRep(\fraka)$ is indeed a set of $f$-representations for
  $(\Red(\fraka)/_\sim$, $d^\fraka)$ in the sense of Definition~\ref{reductionmap-frep-definition},
  we show the following two lemmata, which illustrate how $f$-representations can be obtained
  (``reduction'') and in which way they are unique. These lemmata are crucial to prove that the
  induced map $\fRep(\fraka) \to \G^n / \Lambda$, $(x, t) \mapsto d^\fraka(x) + t$ is a bijection:
  the Reduction Lemma shows that the map is surjective, and the Uniqueness Lemma shows that the map
  is injective.
  
  The first result, the Reduction Lemma, shows that any tuple $(\frakb, (t_1, \dots, t_n)) \in
  \Id(\calO_K) \times \G^n$ can be \emph{reduced} to an $f$-representation. Similar to reducing an
  ideal, this procedure divides by a minimum~$\mu$ of the ideal. The $t_i$\indexplural\ have to be
  adjusted by the valuations of $\mu$. In particular, this result shows that for every
  ideal~$\fraka$, the set $\Red(\fraka)$ is not empty, hence giving another proof of the
  non-emptiness result in Proposition~\ref{distanceInjectivityProp}. In fact, the proof is very
  similar to the proof of that proposition, except that here, we divide by a very specific minimum
  of $\frakb$.
  
  \begin{lemma}[Existence and Reduction]
    \label{reduction-lemma}
    Let $\frakb$ be any ideal equivalent to $\fraka$, and let $t_1, \dots, t_n \in \G$. Then there
    exists a smallest $\ell \in \G$ such that $B_\ell := B(\frakb, (t_1, \dots, t_n, \ell))$
    $\setminus \{ 0 \}$ is non-empty. If $\mu \in B_\ell$ is a smallest element with respect to
    $\le$, then \[ ([\tfrac{1}{\mu} \frakb]_\sim, (t_1 + \nu_{\frakp_1}(\mu), \dots, t_n +
    \nu_{\frakp_n}(\mu))) \in \fRep(\fraka). \]
  \end{lemma}
  
  The proof of Lemma~\ref{reduction-lemma} shows why $\ell$ and $\mu$ exist, as claimed in the
  statement of the lemma.
  
  Before we prove this result, we want to discuss it in the function field case. Let $D =
  \divisor(\frakb) + \sum_{i=1}^n t_i \frakp_i$; then $B_\ell = L(D + \ell \frakp_{n+1}) \setminus
  \{ 0 \}$. In the case of function fields with $\deg \frakp_{n+1} = 1$, He\ss' reduction method as
  described in \cite{hessRR} works by minimizing $\ell$ with $L(D + \ell \frakp_{n+1}) \neq \{ 0
  \}$, then chosing an element $\mu \in L(D + \ell \frakp_{n+1}) \setminus \{ 0 \}$ and replacing
  $D$ by \[ D + \ell \frakp_{n+1} + (\mu) = \divisor\bigl(\tfrac{1}{\mu} \frakb\bigr) + \sum_{i=1}^n
  (t_i + \nu_{\frakp_i}(\mu)) \frakp_i. \] Since $\deg \frakp_{n+1} = 1$, $\dim_k L(D + \ell
  \frakp_{n+1}) = 1$, so the choice of $\le$ does not matter: any other element~$\mu'$ of $B_\ell$
  will yield the same reduced divisor~$D + (\mu')$.
  
  If $\deg \frakp_{n+1} > 1$, the condition that $\mu$ is a smallest element in $B_\ell$ with
  respect to $\le$ ensures that $\mu$ is indeed a minimum of $\frakb$, i.e.\ that $\frac{1}{\mu}
  \frakb$ is reduced in the sense of Definition~\ref{minimumreduceddef}. This shows that the
  procedure described in the lemma generalizes He\ss' reduction.
  
  Note that if we consider the set $X = \bigcup_{\ell \in \G} B_\ell$, then $X$ has a smallest
  element with respect to $\le$, and every such smallest element~$\mu$ will satisfy that $\ell =
  -\nu_{\frakp_{n+1}}(\mu)$ is minimal with $B_\ell \neq \emptyset$: this is ensured by the choice
  of $\le$, which ``prefers'' elements with smaller absolute
  value~$\abs{\placeholder}_{\frakp_{n+1}}$. Hence, we could relax the lemma by not asking that
  $\ell$ is minimal, but just that $B_\ell \neq \emptyset$.
  
  \begin{proof}[Proof of Lemma~\ref{reduction-lemma}.]
    If $\ell \ll 0$, we have $B_\ell = \emptyset$ by the Product Formula.\footnote{Here, a statement
    being true for $x \ll 0$ (respectively, $x \gg 0$) means that there exists some $N$ such that
    the statement holds for all $x \le N$ (respectively, $x \ge N$).} For $\ell \gg 0$, we get that
    $B_\ell\neq \emptyset$ by Riemann's Inequality, respectively, Minkowski's Lattice Point
    Theorem. Choose $\ell \in \G$ minimal such that $B_\ell \neq \emptyset$; in the number field
    case, this is possible since $B_\ell$ is a finite set: hence, if $\ell'$ is chosen such that
    $B_{\ell'}$ is non-empty, we can choose $\ell = -\max\{ \nu_{\frakp_{n+1}}(x) \mid x \in
    B_{\ell'} \}$.
    
    If $K$ is a number field, then $B_\ell$ is a finite set, whence a minimal element with respect
    to $\le$ clearly exists as well. If $K$ is a function field, then the infinite
    valuations~$\nu_\frakp$ for $\frakp \in S$ take on only finitely many values on $B_\ell$ since
    $B_\ell \cup \{ 0 \}$ is a finite-dimensional vector space, whence the existence of $\mu$ is
    clear, too.
    
    If $\ell$ is minimal, we have $-\nu_{\frakp_{n+1}}(\mu) = \ell$ by choice of $\mu$. Moreover, \[
    B(\tfrac{1}{\mu} \frakb, (t_1 + \nu_{\frakp_1}(\mu), \dots, t_n + \nu_{\frakp_n}(\mu), 0)) =
    \tfrac{1}{\mu} B(\frakb, (t_1, \dots, t_n, \ell)) \] and, by choice of $\mu$, we have that $1 =
    \frac{\mu}{\mu}$ lies in this set and is minimal among the non-zero elements with respect to
    $\le$. Hence, by Remark~\ref{frepremarks}~(c), the claim follows.
  \end{proof}
  
  The second result, uniqueness, shows that reducing an $f$-representation will always yield the
  same $f$-representation. This will be utilized in showing that the map $\fRep(\fraka) \to \G^n /
  \Lambda$ is injective: in the proof of Theorem~\ref{infrathm-1}, we will show that if two
  $f$-representations are mapped onto the same element of $\G^n / \Lambda$, then one is a reduction
  in the sense of the Reduction Lemma~\ref{reduction-lemma} of the other.
  
  \begin{lemma}[Uniqueness]
    \label{uniqueness-lemma}
    Let $A := ([\frakb]_\sim, (t_1, \dots, t_n)) \in \fRep(\fraka)$ and let $\mu \in K^*$ such that
    $B := ([\frac{1}{\mu} \frakb]_\sim, (t_1 + \nu_{\frakp_1}(\mu), \dots, t_n +
    \nu_{\frakp_n}(\mu))) \in \fRep(\fraka)$. Then $\abs{\mu}_\frakp = 1$ for all $\frakp \in S$,
    and hence $A = B$.
  \end{lemma}
  
  \begin{proof}
    As $1 \in B(\frac{1}{\mu} \frakb, (t_1 + \nu_{\frakp_1}(\mu), \dots, t_n + \nu_{\frakp_n}(\mu),
    0))$, we get $\mu = \mu \cdot 1 \in \mu B(\frac{1}{\mu} \frakb, (t_1 + \nu_{\frakp_1}(\mu),
    \dots, t_n + \nu_{\frakp_n}(\mu), 0)) \setminus \{ 0 \} = B(\frakb, (t_1, \dots, t_n,
    -\nu_{\frakp_{n+1}}(\mu))) \setminus \{ 0 \}$. Hence, $\mu$ is minimal in $B(\frakb, (t_1,
    \dots, t_n, -\nu_{\frakp_{n+1}}(\mu))) \setminus \{ 0 \}$ with respect to $\le$. By the choice
    of $\le$, it is also minimal in $B(\frakb, (t_1, \dots, t_n, \max\{ 0, -\nu_{\frakp_{n+1}}(\mu)
    \})) \setminus \{ 0 \}$; but then, by the same argument, $1$ is minimal with respect to $\le$ in
    the same set. Thus, we get $\mu \le 1 \le \mu$, which shows that $\abs{\mu}_\frakp = 1$ for
    every $\frakp \in S$.
  \end{proof}
  
  Now we will state our main result in this section, which asserts that the set $\fRep(\fraka)$ as
  given in Definition~\ref{fRepDefForGlobalField} indeed defines a set of $f$-representations for
  $(X^\fraka, d^\fraka)$ in the sense of Definition~\ref{reductionmap-frep-definition}. This can be
  seen as generalizing Proposition~\ref{onedimfRepBijection} for the one-dimensional infrastructure
  case. Note that the set $X^\fraka = \Red(\fraka)/_\sim$ is possibly infinite if $K$ is a function
  field and $k$ is not finite.
  
  \begin{theorem}[Infrastructure, Part~I: Correspondence between $f$-representations and
    $\G^n/\Lambda$]
    \label{infrathm-1}
    The map \[ \Phi^\fraka : \fRep(\fraka) \to \G^n / \Lambda, \qquad ([\frakb]_\sim, t) \mapsto
    d^\fraka(\frakb) + t \] is a bijection, and $([\frakb]_\sim, (0, \dots, 0)) \in \fRep(\fraka)$
    for every $[\frakb]_\sim \in \Red(\fraka) /_\sim$.
  \end{theorem}
  
  Note that this result generalizes the injectivity of $d^\fraka$ in
  Proposition~\ref{distanceInjectivityProp}: for that result, it suffices to note that the map
  $\Red(\fraka)/_\sim \to \fRep(\fraka)$, $[\frakb]_\sim \mapsto ([\frakb]_\sim, (0, \dots, 0))$ is
  an injection.
  
  \begin{proof}[Proof of Theorem~\ref{infrathm-1}.]
    The second part is Remark~\ref{frepremarks}~(b). For the injectivity of $\Phi^\fraka$, let $A =
    ([\frakb]_\sim, (t_1, \dots, t_n)), A' = ([\frakb']_\sim, (t_1', \dots, t_n')) \in
    \fRep(\fraka)$ with $\Phi^\fraka(A) = \Phi^\fraka(A')$. Write $\frakb = \frac{1}{\mu} \fraka$
    and $\frakb' = \frac{1}{\mu'} \fraka$. Then there exists $\varepsilon \in \calO_K^*$ with
    $\Psi(\mu) + (t_1, \dots, t_n) = \Psi(\mu') + (t_1', \dots, t_n') + \Psi(\varepsilon)$. Define
    $\mu'' := \mu^{-1} \mu' \varepsilon$; then $t_i + \nu_{\frakp_i}(\mu'') = t_i'$ and
    $\frac{1}{\mu''} \frakb = \frakb'$, whence by the Uniqueness Lemma~\ref{uniqueness-lemma}, we
    get $A = A'$.
    
    For the surjectivity of $\Phi^\fraka$, let $(t_1, \dots, t_n) + \Lambda \in \G^n /
    \Lambda$. Then by the Reduction Lemma~\ref{reduction-lemma}, there exists $\mu \in \fraka$ such
    that $A'' = ([\frac{1}{\mu} \fraka]_\sim, (t_1 + \nu_{\frakp_1}(\mu), \dots, t_n +
    \nu_{\frakp_n}(\mu))) \in \fRep(\fraka)$. Now $\Phi^\fraka(A'') = \Psi(\mu) + (t_1 +
    \nu_{\frakp_1}(\mu), \dots, t_n + \nu_{\frakp_n}(\mu)) + \Lambda = (t_1, \dots, t_n) + \Lambda$,
    as we wanted to show.
  \end{proof}
  
  So far, we have obtained a set $X^\fraka = \Red(\fraka) /_\sim$ of classes of reduced ideals
  equivalent to $\fraka$, together with a distance map $d^\fraka : X^\fraka \to \G^n / \Lambda$ and
  a set of $f$-representations $\fRep(\fraka) \subset X^\fraka \times \G^n / \Lambda$ for
  $(X^\fraka, d^\fraka)$. This means that the map \[ \Phi^\fraka : \fRep(\fraka) \to \G^n / \Lambda,
  \qquad (x, t) \mapsto d^\fraka(x) + t \] is a bijection, and we know that $(x, 0) \in
  \fRep(\fraka)$ for all $x \in X^\fraka$. This allows us to define a reduction map \[
  \reduce^\fraka : \G^n / \Lambda \to X^\fraka \] for $(X^\fraka, d^\fraka)$ as in the previous
  section, by taking $\reduce^\fraka(v)$ to be the first component of $(\Phi^\fraka)^{-1}(v) \in
  \fRep(\fraka)$. Therefore, assuming $K$ is a global field, $(X^\fraka, d^\fraka, \reduce^\fraka)$
  is an \emph{$n$-dimensional infrastructure}, and we obtain a giant step \[ \gs^\fraka(x, x') :=
  \reduce^\fraka(d^\fraka(x) + d^\fraka(x')), \qquad x, x' \in X^\fraka. \] Moreover, as in the
  one-dimensional case, we can use $\Phi^\fraka$ to turn $\fRep(\fraka)$ into an abelian group by
  pulling back the group operation from $\G^n / \Lambda$: for $A, B \in \fRep(\fraka)$, define \[ A
  \oplus_\fraka B := (\Phi^\fraka)^{-1}(\Phi^\fraka(A) + \Phi^\fraka(B)). \] Then $(\fRep(\fraka),
  \oplus_\fraka)$ is an abelian group isomorphic to $\G^n / \Lambda$ via $\Phi^\fraka$. We denote
  this group operation by $\oplus_\fraka$ and not by $+$ since in the next section, we will equip
  $\bigcup_{\fraka \in \Id(\calO_K)} \fRep(\fraka)$ with a group operation named $+$ which is
  related to the (Arakelov) divisor class group of $K$. Now $(\fRep(\calO_K), \oplus_{\calO_K})$
  will be a subgroup of $\bigcup_{\fraka \in \Id(\calO_K)} \fRep(\fraka)$, but no other
  $(\fRep(\fraka), \oplus_\fraka)$ will be a subgroup. Therefore, we reserve the symbol $+$ for the
  operation defined in the next section.

  If $\fraka = \calO_K$, then the results in the next section allow us to explicitly describe the
  group operation on $\fRep(\calO_K)$. It is essentially ideal multiplication, followed by a
  reduction: if $A = ([\frakb]_\sim, (t_1, \dots, t_n)), B = ([\frakb']_\sim, (t_1', \dots, t_n'))
  \in \fRep(\calO_K)$, we can apply the Reduction Lemma~\ref{reduction-lemma} to $(\frakb \frakb',
  (t_1 + t_1', \dots, t_n + t_n'))$. In case $\fraka \neq \calO_K$, one can still describe the group
  operation~$\oplus_\fraka$ on $\fRep(\fraka)$, but one cannot use simple ideal multiplication since
  $\frakb \frakb'$ will not be in the ideal class of $\fraka$ as soon as $\fraka$ is not a principal
  ideal; and even if $\fraka$ is principal, distances will be added incorrectly. The correct formula
  is given as follows:
  
  \begin{proposition}
    \label{fRepArithmeticOplus}
    Let $A = ([\frakb]_\sim, (t_1, \dots, t_n)), B = ([\frakb']_\sim, (t_1', \dots, t_n')) \in
    \fRep(\fraka)$. Apply the Reduction Lemma~\ref{reduction-lemma} to $(\frakb \frakb' \fraka^{-1},
    (t_1 + t_1', \dots, t_n + t_n'))$, and denote the result by $C$. Then $C \in \fRep(\fraka)$ and
    $A \oplus_\fraka B = C$.
  \end{proposition}
  
  \begin{proof}
    Write $\frakb = \frac{1}{\mu} \fraka$ and $\frakb' = \frac{1}{\mu'} \fraka$. Then $\frakb
    \frakb' \fraka^{-1} = \frac{1}{\mu \mu'} \fraka$ lies in the ideal class of $\fraka$. We can now
    conclude with
    \begin{align*}
      \Phi^\fraka(C) ={} & d^\fraka(\frakb \frakb' \fraka^{-1}) + (t_1 + t_1', \dots, t_n + t_n') \\
      {}={} & d^\fraka(\frakb) + (t_1, \dots, t_n) + d^\fraka(\frakb') + (t_1', \dots, t_n') =
      \Phi^\fraka(A) + \Phi^\fraka(B). \qedhere
    \end{align*}
  \end{proof}
  
  This result is similar to Remark~\ref{frepOneDimArithmetic} and
  Example~\ref{realquadratic-example} in the one-dimensional case: the reduced ideals are multiplied
  and the product is then reduced. In case $\fraka \neq \calO_K$, a correction factor needs to be
  multiplied to the product of the ideals.
  
  In this section, we saw how to construct a set of $f$-representations~$\fRep(\fraka)$ and,
  therefore, a reduction map~$\reduce^\fraka$ for $(X^\fraka, d^\fraka) = (\Red(\fraka)/_\sim,
  d^\fraka)$, thereby turning this pair into an $n$-dimensional infrastructure~$(X^\fraka, d^\fraka,
  \reduce^\fraka)$. We also saw how to explicitly compute the group operation induced by the
  bijection $\fRep(\fraka) \to \G^n / \Lambda$ in terms of ideal multiplication followed by
  reduction. This also shows how the giant step operation~$\gs([\frakb]_\sim, [\frakb']_\sim)$ can
  be computed, by ignoring the $t$-part of the resulting $f$-representation~$([\frakb]_\sim, 0)
  \oplus_\fraka ([\frakb']_\sim, 0)$. This operation generalizes Shanks' original approach as
  sketched in Example~\ref{realquadratic-example}.
  
  \section{Relation to the Divisor Class Group}
  \label{relationtodcg}
  
  In this section, we want to relate the set of all $f$-representations, \[ \fRep(K) :=
  \bigcup_{\fraka \in \Id(\calO_K)} \fRep(\fraka), \] to the (Arakelov) divisor class
  group~$\Pic^0(K)$. In case $K$ is a number field, or in case $K$ is a function field and $\deg
  \frakp_{n+1} = 1$, we obtain an isomorphism $\fRep(K) \to \Pic^0(K)$. In case $K$ is a function
  field and $\deg \frakp_{n+1} > 1$, we can identify a subset of $\fRep(K)$ with $\Pic^0(K)$. We
  show that we can then extend $\Pic^0(K)$ to obtain a group which is isomorphic to
  $\fRep(K)$. Finally, we show how to perform effective arithmetic in $\fRep(K)$.
  
  To motivate the fact that there is a relationship between our infrastructures $(X^\fraka,
  d^\fraka, \reduce^\fraka)$ together with $\fRep(\fraka)$ and the (Arakelov) divisor class group
  $\Pic^0(K)$, we first consider the aforementioned special case. Assume for a moment that $\deg
  \frakp_{n+1} = 1$, or that $K$ is a number field. In this case, we have the short exact
  sequence \[ \xymatrix{ 0 \ar[r] & T \ar[r] & \Pic^0(K) \ar[r] & \Pic(\calO_K) \ar[r] & 0, } \] and
  we have $T \cong \G^n / \Lambda$. Moreover, we have a representation of $\G^n / \Lambda$ by
  $\fRep(\fraka)$ for every $\fraka \in \Id(\calO_K)$, which consists of all $f$-representations
  whose reduced ideals range over all reduced ideals in the ideal class of $\fraka$. By the short
  exact sequence, clearly the (Arakelov) divisor class group~$\Pic^0(K)$ is covered by
  $\abs{\Pic(\calO_K)}$~copies of $\G^n / \Lambda$, whence one might hope that $\Pic^0(K)$ can be
  described in a nice way using $\fRep(K) = \bigcup_{\fraka \in \Id(\calO_K)} \fRep(\fraka)$. This
  turns out to be the case. In fact, Paulus and R\"uck already showed this for the special case of
  the infrastructure obtained from a real hyperelliptic curve in \cite{paulus-rueck}.
  
  In the general case, i.e.\ if $\deg \frakp_{n+1}$ is not necessarily $1$, $T$ can be embedded into
  $\G^n / \Lambda$, but might not cover the entire set, and the map $\Pic^0(K) \to \Pic(\calO_K)$
  might not be surjective. This can only happen in the function field case. It would be desirable to
  have a short exact sequence
  \begin{equation}\label{desiredSES}\tag{$\ast$}
    \xymatrix{ 0 \ar[r] & \G^n / \Lambda \ar[r] & ? \ar[r] & \Pic(\calO_K) \ar[r] & 0 }
  \end{equation}
  for all function fields, into which the exact sequence
  
  \begin{equation}\label{simplerSES}\tag{$\ast\ast$}
    \xymatrix{ 0 \ar[r] & T \ar[r] & \Pic^0(K) \ar[r] & \Pic(\calO_K) }
  \end{equation}
  embeds in a natural way. If $D$ is a divisor with $\deg D \neq 0$, one obtains an exact
  sequence \[ \xymatrix@C-0.3cm{ 0 \ar[r] & \Pic^0(K) \ar[r] & \Pic(K) / \ggen{ [D] } \ar[r] & (\deg
  \frakp \mid \frakp \in \calP_K) \Z / (\deg D) \Z \ar[r] & 0. } \] For the right choice of $D$, we
  obtain that $\Pic(K) / \ggen{ [D] }$ is the right replacement for the ``$?$'' in
  Equation~\eqref{desiredSES}. The exact relationship between the exact sequences in
  Equations~\eqref{desiredSES} and \eqref{simplerSES} will be described later in
  Proposition~\ref{commutativediagramprop}.
  
  We now state the main result for this section, which identifies the set of $f$-representations
  with the Arakelov divisor class group, or with an extension of $\Pic^0(K)$.
  
  \begin{theorem}[Infrastructure, Part~II: Relating $f$-representations to the divisor class
    group]\hfill
    \label{infrathm-2}
    \begin{enuma}
      \item Let $K$ be a number field. Then the following map is a bijection:
      \begin{align*}
        \Phi :{} & \fRep(K) \to \Pic^0(K), \\
        & ([\frakb]_\sim, (t_1, \dots, t_n)) \mapsto \biggl[ \divisor(\frakb) + \sum_{i=1}^n t_i
        \frakp_i - \tfrac{\deg \divisor(\frakb) + \sum_{i=1}^n t_i \deg \frakp_i}{\deg \frakp_{n+1}}
        \frakp_{n+1} \biggr]
      \end{align*}
      \item Let $K$ be a function field. Then the following map is a bijection:
      \begin{align*}
        \Phi :{} & \fRep(K) \to \Pic(K) / \ggen{ [\frakp_{n+1}] }, \\
        & ([\frakb]_\sim, (t_1, \dots, t_n)) \mapsto \biggl[ \divisor(\frakb) + \sum_{i=1}^n t_i
        \frakp_i \biggr] + \ggen{ [\frakp_{n+1}] }
      \end{align*}
    \end{enuma}
  \end{theorem}
  
  \begin{proof}
    Clearly, the divisors in the definition of $\Phi$ in the number field case are all of
    degree~zero. Hence, one can treat both cases at the same time by ignoring the valuations of the
    divisors at $\frakp_{n+1}$. First, note that the maps are well-defined, since if $\frakb$ is
    replaced by $h \frakb$ for some $h \in K^*$ with $\abs{h}_\frakp = 1$ for all $\frakp \in S$,
    then $\divisor(\frakb)$ is replaced by $\divisor(\frakb) - (h) = \divisor(h \frakb)$.
    
    To show injectivity, let $A = ([\frakb]_\sim, (t_1, \dots, t_n))$ and $A' = ([\frakb']_\sim,
    (t_1', \dots, t_n')) \in \fRep(K)$ with $\Phi(A) = \Phi(A')$, i.e.\ let $h \in K^*$ and $\ell
    \in \G$ with \[ \divisor(\frakb) + \sum_{i=1}^n t_i \frakp_i = \divisor(\frakb') + \sum_{i=1}^n
    t_i' \frakp_i + (h) + \ell \frakp_{n+1}. \] This gives $\frac{1}{h} \frakb' = \frakb$ and $t_i =
    t_i' + \nu_{\frakp_i}(h)$. But then, $A = ([\frac{1}{h} \frakb'], (t_1' + \nu_{\frakp_1}(h),
    \dots, t_n' + \nu_{\frakp_n}(h)))$, whence by the Uniqueness Lemma~\ref{uniqueness-lemma} we get
    $\abs{h}_\frakp = 1$ for every $\frakp \in S$. But this implies $A = A'$. Therefore, $\Phi$ is
    injective.
    
    For surjectivity, let $[D] \in \Pic^0(K)$, respectively, $[D] \in \Pic(K) / \ggen{
    [\frakp_{n+1}] }$. Write $D = \divisor(\fraka) + \sum_{i=1}^n t_i \frakp_i + \ell \frakp_{n+1}$
    for $\fraka \in \Id(\calO_K)$, $t_1, \dots, t_n, \ell \in \G$. By Reduction
    Lemma~\ref{reduction-lemma}, there exists a $\mu \in \frakb$ such that $B = ([\frac{1}{\mu}
    \fraka]_\sim, (t_1 + \nu_{\frakp_1}(\mu), \dots, t_n + \nu_{\frakp_n}(\mu))) \in \fRep(K)$, and,
    up to $\frakp_{n+1}$, the divisor in $\Phi(B)$ equals \[ \divisor\bigl(\tfrac{1}{\mu}
    \fraka\bigr) + \sum_{i=1}^n (t_i + \nu_{\frakp_i}(\mu)) \frakp_i = \divisor(\fraka) +
    \sum_{i=1}^n t_i \frakp_i + (\mu) - \nu_{\frakp_{n+1}}(\mu) \frakp_{n+1}, \] i.e.\ $\Phi(B) =
    [D]$.
  \end{proof}
  
  Note that this gives, in particular, an embedding of $\Red(K)/_\sim$ into $\Pic^0(K)$,
  respectively, $\Pic(K) / \ggen{ [\frakp_{n+1}] }$, where $\Red(K) = \bigcup_{\fraka \in
  \Id(\calO_K)} \Red(\fraka)$. In the case of number fields, Schoof gave a similar embedding in
  \cite{schoofArakelov}; more precisely, he embedded $\Red(K)$ in the \emph{oriented} Arakelov
  divisor class group~$\widetilde{\Pic}{}^0(K)$, which is a cover of $\Pic^0(K)$. Moreover, his
  embedding assigns different valuations for the infinite places. Our embedding has the advantage
  that it works in a very similar way for both number fields and function fields. In the case of
  real hyperelliptic function fields, our embedding is the same as the one by Paulus and R\"uck
  \cite[Theorem~4.2]{paulus-rueck}. Moreover, in part~(b) of the theorem, the divisor whose class is
  taken is reduced along $\frakp_{n+1}$ in the sense of He\ss\ \cite{hessRR}. In case $\deg
  \frakp_{n+1} = 1$, this shows that $f$-representations directly correspond to \emph{arbitrary}
  reduced divisors in the sense of He\ss\ which are reduced along $\frakp_{n+1}$.
  
  Finally, note that if we denote by $\G[\frakp_{n+1}]$ the 1-parameter subgroup $\{ [g
  \frakp_{n+1}] \mid g \in \G \}$ of $\Pic(K)$ in case $K$ is a number field, then we can identify
  $\Pic^0(K)$ with $\Pic(K) / \G[\frakp_{n+1}]$. Hence, we can write $\Phi$ as
  \begin{align*}
    \Phi :{} & \fRep(K) \to \Pic(K) / \G[\frakp_{n+1}], \\
    & ([\frakb]_\sim, (t_1, \dots, t_n)) \mapsto \biggl[ \divisor(\frakb) + \sum_{i=1}^n t_i
    \frakp_i \biggr] + \G[\frakp_{n+1}]
  \end{align*}
  for both number fields and function fields. Thus, Theorem~\ref{infrathm-2} completely unifies the
  number field and function field scenarios.
  
  Before describing how the group operation on $\fRep(K)$ induced by the one on
  $\Pic(K)/\G[\frakp_{n+1}]$ can be computed, we want to state a result on the interrelations
  between all aforementioned groups. For that, we first make clear how the map $T \to \G^n /
  \Lambda$ is defined. Assume that $T = \Div_\infty^0(K) / (\calO_K^*/k^*)$, where $\calO_K^*/k^*$
  is embedded into $\Div_\infty^0(K)$ by forming principal divisors. Then we obtain a map $T
  \hookrightarrow \G^n / \Lambda$ by mapping the class of $\sum_{\frakp \in S} t_\frakp \frakp$ to
  $(t_{\frakp_1}, \dots, t_{\frakp_n}) + \Lambda$. This map is clearly injective. In case $K$ is a
  number field or if $\deg \frakp_{n+1} = 1$, it is surjective as well.
  
  \begin{proposition}
    \label{commutativediagramprop}
    The diagram \[ \xymatrix@R-0.3cm@C-0.2cm{ 0 \ar[r] & T \ar[r] \ar@<-2pt>@{^(->}[d] & \Pic^0(K)
    \ar[r] \ar@<-2pt>@{^(->}[d] & \Pic(\calO_K) \ar@{=}[dd] & \\ & \G^n / \Lambda
    \ar[d]_\cong^{(\Phi^\fraka)^{-1}} & \Pic(K) / \G[\frakp_{n+1}] \ar[d]^{\Phi^{-1}}_{\cong} & & \\
    0 \ar[r] & \fRep(\calO_K) \ar[r] & \fRep(K) \ar[r] & \Pic(\calO_K) \ar[r] & 0 } \] commutes. In
    case $K$ is a function field, the image of $T$ in $\G^n/\Lambda$ is the set \[ \biggl\{ (t_i)_i
    + \Lambda \in \G^n/\Lambda \,\biggm|\, \deg \frakp_{n+1} \text{ divides } \sum_{i=1}^n t_i \deg
    \frakp_i \biggr\}, \] and the image of $\Pic^0(K)$ in $\fRep(K)$ is the set \[ \biggl\{
    ([\fraka]_\sim, (t_1, \dots, t_n)) \in \fRep(K) \,\biggm|\, \deg \frakp_{n+1} \text{ divides }
    \deg \divisor(\fraka) + \sum_{i=1}^n t_i \deg \frakp_i \biggr\}. \]
  \end{proposition}
  
  Proposition~\ref{commutativediagramprop} shows in particular that the group
  operation~$\oplus_{\calO_K}$ on $\fRep(\calO_K)$ defined in the last section is identical to the
  group operation~$+$ obtained from the group operation on $\Pic(K)/\G[\frakp_{n+1}]$ restricted to
  the subset $\fRep(\calO_K)$. Hence, we are able to relate two group operations which were defined
  quite differently: $\oplus_{\calO_K}$ is defined by pulling back the addition from $\G^n /
  \Lambda$, and $+$ is defined by pulling back the addition from $\Pic(K)/\G[\frakp_{n+1}]$.
  
  \begin{proof}[Proof of Proposition~\ref{commutativediagramprop}.]
    We first show that the left square commutes. For that, we compare the maps $T \to \G^n/\Lambda
    \to \fRep(\calO_K) \to \fRep(K) \to \Pic(K)/\G[\frakp_{n+1}]$ with $T \to \Pic^0(K) \to
    \Pic(K)/\G[\frakp_{n+1}]$. Let the class of $D = \sum_{i=1}^{n+1} t_i \frakp_i$ be an element of
    $T$. Then it is mapped to $(t_1, \dots, t_n) + \Lambda$ in $\G^n/\Lambda$ and to an
    $f$-representation $A = ([\frac{1}{\mu} \calO_K]_\sim, (t'_1, \dots, t'_n)) \in \fRep(\calO_K)$
    such that
    \begin{equation}\tag{$\ast$}
      \Phi^{\calO_K}(A) = \Psi(\mu) + (t'_1, \dots, t'_n) + \Lambda = (t_1, \dots, t_n) + \Lambda.
    \end{equation}
    This in turn is mapped to the class of $\divisor(\frac{1}{\mu} \calO_K) + \sum_{i=1}^n t'_i
    \frakp_i$ in $\Pic(K)/\G[\frakp_{n+1}]$. Hence, we evaluated the class of $D$ along the first
    composition of maps.
    
    Now $D$ is rationally equivalent to $\sum_{i=1}^{n+1} t_i \frakp_i + (\mu)$. The finite part of
    this divisor is $\divisor(\frac{1}{\mu} \calO_K)$. The valuation of this divisor at $\frakp_i$
    is $t_i + \nu_{\frakp_i}(\mu)$ for $1 \le i \le n$, and $(t_i + \nu_{\frakp_i}(\mu))_i + \Lambda
    = (t'_i)_i + \Lambda$ by ($\ast$). But this means that $[D] = [D - \mu] = [\divisor(\fraka) +
    \sum_{i=1}^n t'_i \frakp_i + \ell \frakp_{n+1}]$ in $\Pic^0(K)$ for suitable $\ell \in \G$,
    whence the first square commutes.
    
    To see that the second square commutes, note that if $[D] \in \Pic^0(K)$ with $D =
    \divisor(\fraka) + \sum_{i=1}^{n+1} t_i \frakp_i$, then $[D]$ maps to the ideal class of
    $\fraka$ in $\Pic(\calO_K)$. Now the $f$-representation representing $[D] + \G[\frakp_{n+1}]$
    can be found by reducing $(\fraka, (t_1, \dots, t_n))$, yielding the ideal part $[\frac{1}{\mu}
    \fraka]_\sim$ for some $\mu \in \calE(\fraka)$. But the resulting $f$-representation is mapped
    to the ideal class of $\frac{1}{\mu} \fraka$ in $\Pic(\calO_K)$, which is the same as the ideal
    class of $\fraka$. Therefore, the second square also commutes.
    
    Finally, in case $K$ is a function field, the equalities for the images of $T$ in $\G^n /
    \Lambda$ and $\Pic^0(K)$ in $\fRep(K)$ follow from the fact that divisors representing elements
    of $T$ and $\Pic^0(K)$ must have degree~zero.
  \end{proof}
  
  It turns out that the group operations in $\Pic^0(K)$, respectively, $\Pic(K) / \ggen{
  [\frakp_{n+1}] }$, can be described in a nice way using $f$-representations. This directly
  generalizes the arithmetic in $(\fRep(\calO_K), \oplus_{\calO_K})$ as described in
  Proposition~\ref{fRepArithmeticOplus}. Note that this is \emph{not} related to the arithmetic in
  $(\fRep(\fraka), \oplus_\fraka)$ for $\fraka \neq \calO_K$.
  
  The following theorem describes how the group operations on $\fRep(K)$ can be effectively
  computed.
  
  \begin{theorem}[Infrastructure, Part~III: Computing the group operation]
    \label{infrathm-3}
    Let $A = ([\frakb]_\sim, (t_1, \dots, t_n)), A' = ([\frakb']_\sim,$ $(t_1', \dots, t_n')) \in
    \fRep(K)$.
    \begin{enuma}
      \item There exists a minimal $\ell \in \G$ such that $B_\ell := B(\frakb \frakb', (t_1 + t_1',
      \dots, t_n + t_n', \ell)) \setminus \{ 0 \}$ is non-empty; if $\mu$ is a smallest element with
      respect to $\le$ in $B_\ell$, we get $B := ([\frac{1}{\mu} \frakb \frakb']_\sim, (t_1 + t_1' +
      \nu_{\frakp_1}(\mu), \dots, t_n + t_n' + \nu_{\frakp_n}(\mu))) \in \fRep(K)$ and $\Phi(A) +
      \Phi(A') = \Phi(B)$.
      \item There exists a minimal $\ell \in \G$ such that $B_\ell := B(\frakb^{-1}, (-t_1, \dots,
      -t_n, \ell)) \setminus \{ 0 \}$ is non-empty; if $\mu$ is a smallest element with respect to
      $\le$ in $B_\ell$, we get $C := ([\frac{1}{\mu} \frakb^{-1}]_\sim, (-t_1 +
      \nu_{\frakp_1}(\mu), \dots, -t_n + \nu_{\frakp_n}(\mu))) \in \fRep(K)$ and $-\Phi(A) =
      \Phi(C)$.
    \end{enuma}
  \end{theorem}
  
  The main parts of this lemma were already shown in Lemma~\ref{reduction-lemma}, namely that $B$
  and $C$ are indeed $f$-representations. The claims $\Phi(A) + \Phi(A') = \Phi(B)$ and $-\Phi(A) =
  \Phi(C)$ follows from the fact that $\divisor : \Id(\calO_K) \to \Div(K)$ is a group homomorphism
  as well as from the definitions of $\Phi$ and of the group operation on $\Pic^0(K)$, respectively,
  $\Pic(K) / \ggen{ [\frakp_{n+1}] }$.
  
  Note that this ``reduction'' step, namely minimizing $\ell$, and then minimizing $\mu$ with
  respect to $\le$ if necessary, is essentially the same what is used for arithmetic on
  hyperelliptic and superelliptic curves \cite{galbraith-paulus-smart}, and for He\ss' arithmetic in
  function fields with $\deg \frakp_{n+1} = 1$ \cite{hessRR}; compare the discussion following the
  Reduction Lemma~\ref{reduction-lemma}. This is not very surprising, since as we already mentioned,
  $f$-representations are another representation of divisors reduced along $\frakp_{n+1}$.
  
  We have seen that all infrastructures $(\Red(\fraka)/_\sim, d^\fraka)$ in $K$, and their
  corresponding $f$-representations $\fRep(\fraka)$, can be combined to the set of \emph{all}
  $f$-representations $\fRep(K)$, which parameterizes the (Arakelov) divisor class group $\Pic(K) /
  \G[\frakp_{n+1}] \supseteq \Pic^0(K)$ using the bijection~$\Phi$. Moreover, we have seen how the
  group structure on $\fRep(K)$ induced by the one on the (Arakelov) divisor class group can be
  computed in terms of $f$-representations only; this is essentially ideal multiplication followed
  by reduction, hence generalizing Shanks' giant steps. Using Corollary~\ref{idealcomparecorr} and
  Proposition~\ref{idealcompareprop} we are able to compare $f$-representations. Therefore, we can
  represent the (Arakelov) divisor class group using $f$-representations and use them to perform
  effective arithmetic.
  
  \section{Computations Using $f$-Representations}
  \label{computationsfprep}
  
  Infrastructures not only represent an interesting algebraic concept, but
  $f$-re\-pre\-sen\-ta\-tions lend themselves very well to computation and lead to efficient
  algorithms for computing fundamental units in global function fields. They require only limited
  storage and allow for efficient giant step computation, as documented in this section. Further
  evidence supporting the suitability of $f$-representations for computation is provided with three
  non-trivial numerical examples. Proofs of these results and a more detailed discussion of
  implementation go beyond the scope of this work and are the subject of a forth-coming paper.
  
  We begin with a result on the size of $f$-representations, which in the function field case is
  identical to a result by He\ss\ in \cite[Section~8]{hessRR}. In the number field case, it
  generalizes a result by Schoof \cite[Proposition~7.2~(i)]{schoofArakelov} to $f$-representations;
  his result is slightly stronger than the well-known inequality $1 \le \Norm_{K/\Q}(\fraka^{-1})
  \le \sqrt{\abs{\Delta}}$ for $\fraka \in \Red(K)$, where $\Delta$ is the discriminant of $K$.
  
  Remember that $\deg \divisor(\fraka) = -\log \Norm_{K/\Q}(\fraka)$ if $K$ is a number field, and
  $\deg \divisor(\fraka) = -\deg \Norm_{K/k(x)}(\fraka)$ if $K$ is a function field.
  
  \begin{proposition}
    \label{frepdegreeinbothcases:body}
    Let $([\fraka]_\sim, (t_i)_i) \in \fRep(K)$. Then $\divisor(\fraka) \ge 0$ and $t_i \ge 0$ for
    $1 \le i \le n$. If $K$ is a function field, let $g$ be its genus, and if $K$ is a number field,
    let $\Delta$ be its discriminant and $2 s$ its number of complex embeddings. Then \[ 0 \le \deg
    \divisor(\fraka) + \sum_{i=1}^n t_i \deg \frakp_i \le \begin{cases} g + (\deg \frakp_{n+1} - 1)
      & \text{if } K \text{ is a function field,} \\ \tfrac{1}{2} \log \abs{\Delta} - s \log
      \tfrac{\pi}{2} & \text{if } K \text{ is a number field.} \end{cases} \] %\qed
  \end{proposition}
  
  This shows that not only the norm of the integral ideal~$\fraka^{-1}$ as well as the positive
  integers $t_i$ are bounded, but a linear combination of these values with positive coefficients is
  bounded. As shown by Paulus and R\"uck \cite{paulus-rueck}, this bound is sharp in the case of
  real quadratic function fields.
  
  \begin{proof}[Proof of Proposition~\ref{frepdegreeinbothcases:body}.]
    Let $D = \divisor(\fraka) + \sum_{i=1}^n t_i \frakp_i$. Then $B(\fraka, (t_1, \dots, t_n, 0)) =
    L(D)$ contains $k$ and $L(D - \frakp_{n+1}) = B(\fraka, (t_1, \dots, t_n, -\varepsilon)) = 0$
    for every $\varepsilon > 0$, $\varepsilon \in \G$. The inclusion shows $D \ge 0$ as $1 \in k$,
    whence $\divisor(\fraka) \ge 0$ and $t_i \ge 0$, $1 \le i \le n$.
    
    If $K$ is a function field of genus~$g$, by Riemann's Inequality \[ 0 = \dim_k L(D -
    \frakp_{n+1}) \ge 1 - g + \deg \divisor(\fraka) + \sum_{i=1}^n t_i \deg \frakp_i - \deg
    \frakp_{n+1}; \] therefore $\deg \divisor(\fraka) + \sum_{i=1}^n t_i \deg \frakp_i \le g - 1 +
    \deg \frakp_{n+1}$.
    
    If $K$ is a number field with $2 s$~complex embeddings and discriminant $\Delta$, we have
    $B(\fraka, (t_1, \dots, t_n, -\varepsilon)) \neq \{ 0 \}$ for $\varepsilon > 0$ if
    \[ e^{-\varepsilon \deg \frakp_{n+1}} \prod_{i=1}^n e^{t_i \deg \frakp_i} > \bigl(
    \tfrac{2}{\pi} \bigr)^s \sqrt{\abs{\Delta}} \Norm_{K/\Q}(\fraka) \] by Minkowski's Lattice Point
    Theorem \cite[Theorem~5.3]{neukirch}. Hence, we must have \[ \exp\biggl( \sum_{i=1}^n t_i \deg
    \frakp_i - \varepsilon \deg \frakp_{n+1} \biggr) \le \bigl( \tfrac{2}{\pi} \bigr)^s
    \sqrt{\abs{\Delta}} e^{-\deg \divisor(\fraka)}. \] Solving for $\deg \divisor(\fraka) +
    \sum_{i=1}^n t_i \deg \frakp_i$ and considering that this is true for all $\varepsilon > 0$
    yields the claim.
  \end{proof}
  
  To represent a reduced ideal, one can use a Hermite normal form representation with respect to a
  fixed integral basis as described in \cite{cohen}. This allows to represent a fractional ideal
  with a unique binary representation. In the number field case, C.~Thiel showed in
  \cite[Corollary~3.7]{thiel-comprep} that one can represent a reduced ideal in a number field of
  degree~$d$ and discriminant~$\Delta$ with at most $(d^2 + 1) \log_2 \sqrt{\abs{\Delta}}$~bits. For
  function fields, we obtain:
  
  \begin{proposition}
    \label{reducedidealsize:body}
    Let $K$ be a function field. Assume that elements of $k$ can be represented by $\calO(\log
    q)$~bits. Then $f$-representations can be represented by $\calO\bigl(d^2 (g + \deg \frakp_{n+1}
    - 1) \log q \bigr)$~bits. \qed
  \end{proposition}
  
  We will provide a more precise statement as well as a proof in a subsequent paper.
  
  Using a technique similar to He\ss' algorithm for computing Riemann-Roch spaces \cite{hessRR}, we
  implemented $f$-representations for function fields. We made the assumption that $\deg
  \frakp_{n+1} = 1$ to ensure that we can quickly compare $f$-representations by their binary
  representation. We added to our implementation an algorithm by Buchmann and A.~Schmidt
  \cite{buchmann-schmidt} to compute the relation lattice~$\Lambda$ of the elements $(g_1, \dots,
  g_n)$ in $\fRep(\calO_K)$, where $g_i = (\Phi^{\calO_K})^{-1}(e_i)$ if $e_i \in \Z^n$ is the
  $i$-th standard unit vector; note that this is a system of generators of $\fRep(\calO_K)$. This
  lattice equals the unit lattice as defined in Section~\ref{reducedideals}. Since the
  Buchmann-Schmidt algorithm is of baby-step giant-step type and requires $\calO(n
  \sqrt{\abs{\fRep(\calO_K)}})$~group operations and $\calO(\sqrt{\abs{\fRep(\calO_K)}})$~storage of
  group elements, we therefore implemented an algorithm which computes the unit lattice of a global
  function field with at least one infinite place of degree~one in $\calO(\sqrt{R})$~infrastructure
  operations using~$\calO(\sqrt{R})$ storage (assuming $[K : k(x)] = \calO(1)$). This can be seen as
  a generalization of Shanks' baby-step giant-step algorithm for computing the unit lattice for a
  real quadratic number field \cite{shanks-infra}, or of Buchmann's baby-step giant step algorithm
  for computing the unit lattice of an arbitrary number field \cite{buchmann-habil}.
  
  Our algorithm was implemented in C++ using NTL. It currently relies on MAGMA for computation of
  integral bases and information on the infinite places. We present three numerical examples that
  were obtained using our algorithm. We compared the output of our program with MAGMA's built-in
  function \verb"Regulator()"; this function apparently uses He\ss' subexponential algorithm for
  computation of the divisor class group \cite{hessDiss}. We applied both our algorithm and MAGMA to
  the function fields of many curves. As an example, we want to present three curves:
  \begin{enum1}
%    \item $y^3 = x^3 + 3 x^2 + 2 x + 3$ over $\F_{10\,009}$; the function field has genus~$1$, three
%    infinite places of degree~$1$ and regulator~$10\,192$;
%    
%% Current total memory usage: 2042.3MB, failed memory request: 161.3MB
%% System error: Out of memory.
%% All virtual memory has been exhausted so Magma cannot perform this statement.
%% Total time: 14036.030 seconds, Total memory usage: 2042.31MB
%% 
%% top: 233:51 minutes  ->  3.90 h
%    \item $y^3 = (x^2 + x + 11) x^2 (x - 1)^2$ over $\F_{10\,009}$; the function field has
%    genus~$2$, three infinite places of degree~$1$ and regulator~$34\,550\,833$;
%
%% Current total memory usage: 1973.6MB, failed memory request: 94.2MB
%% System error: Out of memory.
%% All virtual memory has been exhausted so Magma cannot perform this statement.
%% Total time: 22206.799 seconds, Total memory usage: 1973.61MB
%% 
%% 6.17 h, 1.9 GB
    \item $y^3 = (x + 1) y^2 - (123 x^3 - 423 x^2 + 948 x - 1) y + (13 x^2 + 3123 x + 11) x^2$ over
    $\F_{1009}$; the function field has genus~$3$, two infinite places of degree~$1$ (so unit
    rank~$1$) and regulator~$496\,804\,315$;
    
% BF := GF(1009);
% R<x> := FunctionField(BF);
% P<Y> := PolynomialRing(R);
% F<y> := FunctionField(Y^3 - (x + 1) * Y^2 + (123 * x^3 - 423 * x^2 + 948 * x - 1) * Y - (13 * x^2 + 3123 * x + 11) * x^2);
% O := MaximalOrderFinite(F);
% time Regulator(O);
% 
% 496804315
% Time: 7298.740
% > quit;
% 
% Total time: 7299.750 seconds, Total memory usage: 102.16MB
%             5410.139 seconds, Total memory usage: 102.67MB (encserver1)
%             5594.699 seconds, Total memory usage: 100.09MB (encserver1)
%             5657.359 seconds, Total memory usage: 100.09MB (encserver1)
%             5269.519 seconds, Total memory usage: 104.22MB (encserver1)
%             5347.349 seconds, Total memory usage: 99.06MB (encserver1)
%             5465.189 seconds, Total memory usage: 100.61MB (encserver1)
%             5595.309 seconds, Total memory usage: 102.16MB (encserver1)
%             5322.769 seconds, Total memory usage: 99.06MB (encserver1)
%             5356.199 seconds, Total memory usage: 99.58MB (encserver2)
%
% 2.03 h, 1.50 h, 1.55 h, 1.57 h, 1.46 h, 1.49 h, 1.52 h, 1.55 h, 1.48 h, 1.49 h

    \item $y^8 = 81 (x+2)^2 (x - 3)^3 (x + 1)^3$ over $\F_{1009}$; the function field has genus~$3$,
    eight infinite places of degree~$1$ (so unit rank~$7$) and regulator~$62\,322\,365$;

% BF := GF(1009);
% R<x> := FunctionField(BF);
% P<Y> := PolynomialRing(R);
% F<y> := FunctionField(Y^8 - 81*(x+2)^2 * (x - 3)^3 * (x + 1)^3);
% O := MaximalOrderFinite(F);
% time Regulator(O);
% 
% 62322365
% Time: 633239.180
% > quit;
% 
% Total time: 633240.219 seconds, Total memory usage: 120.02MB (encserver2)
%              12103.639 seconds, Total memory usage: 126.20MB
%             236719.459 seconds, Total memory usage: 120.53MB (encserver2)
%             744244.069 seconds, Total memory usage: 118.98MB (encserver2)
%              15244.199 seconds, Total memory usage: 127.23MB (encserver1)
%             479483.559 seconds, Total memory usage: 120.53MB (encserver2)
%             776314.030 seconds, Total memory usage: 120.02MB (encserver1)
%
% 7.3 d, 3.4 h, 2.7 d, 8.6 d, 4.2 h, 5.5 d, 8.9 d
    \item $(2+\alpha) (y^4 - y^2) + \frac{1 - \alpha}{x} (y^3 + y^2) + \frac{1}{1 - (1 + \alpha^2)
    x} y = 12 \frac{\alpha - x}{x}$ over $\F_{31^2} = \F_{31}[\alpha]$ with $\alpha^2 + 29 \alpha +
    3 = 0$; the function field has genus~$3$, two infinite places of degree~$1$, one infinite place
    of degree~$2$ (so unit rank~$2$), and regulator~$896\,118\,755$.

% BF<a> := GF(31^2);
% a^2 + 29 * a + 3;
% R<x> := FunctionField(BF);
% P<Y> := PolynomialRing(R);
% F<y> := FunctionField((2+a) * (Y^4 - Y^2) + (1 - a)/x * (Y^3 + Y^2) + Y/(1 - (1 + a^2) * x) - 12);
% O := MaximalOrderFinite(F);
% time Regulator(O);
% 
% Total time: 140.210 seconds, Total memory usage: 110.47MB
%             151.060 seconds, Total memory usage: 109.44MB
%             139.760 seconds, Total memory usage: 110.47MB
%             150.450 seconds, Total memory usage: 110.47MB
%             148.590 seconds, Total memory usage: 109.44MB
%             147.360 seconds, Total memory usage: 109.44MB
%             139.449 seconds, Total memory usage: 109.95MB
%             141.230 seconds, Total memory usage: 109.95MB
%             149.790 seconds, Total memory usage: 109.95MB
%             141.909 seconds, Total memory usage: 109.95MB
%             139.550 seconds, Total memory usage: 109.95MB
%             145.000 seconds, Total memory usage: 109.44MB
%
% 2.3 m, 2.5 m, 2.3 m, 2.5 m, 2.5 m, 2.5 m, 2.3 m, 2.4m, 2.5 m, 2.4 m, 2.3 m, 2.4 m
  \end{enum1}
  These fields show that our implementation is not restricted to curves of special form, small unit
  rank, or prime fields. We also do not require all infinite places to have degree~one.
  
  For the first field, MAGMA ran 1.4 to 2.0~hours (in ten different runs) and required between 99~MB
  and 104~MB of memory to compute the regulator. For the second field, MAGMA's running time varied
  dramatically between 3.4~hours and 8.9~days in seven runs, with an average of 4.8~days; the memory
  consumption ranged between 119~MB and 127~MB, where usually the memory usage was around 120~MB, and only
  spiked up to 127~MB for the two runs which needed only a few hours. For the last field, MAGMA worked 2.4~minutes and
  required 110~MB of memory (with minimal variations in twelve runs). On the same machine, our
  implementation was able to compute the regulator in 13.6~minutes for the first field using 46~MB
  of memory, in 9.2~hours for the second field using 97~MB of memory, and in 11.6~hours for the last
  field using 313~MB of memory.
  
  Note that our implementation is not very optimized and in a very general form. Nonetheless, this
  demonstrates that the techniques developed in this paper can be used for computation, and even
  outperform the built-in functions of MAGMA in certain cases. The latter is not surprising, since
  the algorithm MAGMA apparently uses is designed for small constant fields and characteristics, and
  for such function fields is in general much faster than our implementation.
  
  \section{Conclusion}
  \label{conclusion}
  
  We presented a concise interpretation of the infrastructure in a global field, by considering a
  finite set $X^\fraka$, consisting of equivalence classes of reduced ideals in the ideal class of
  $\fraka$, and a distance map $d^\fraka : X^\fraka \to \G^n / \Lambda$, where $\Lambda$ is
  essentially $\calO_K^*/k^*$. We have shown how one can find a reduction map $\reduce^\fraka : \G^n
  / \Lambda \to X^\fraka$ by providing a set of $f$-representations. This generalizes
  one-dimensional infrastructures, and in particular Shanks' original approach and its
  interpretation by Lenstra \cite{lenstra-infrastructure}. 
  
  Considering all infrastructures $(X^\fraka, d^\fraka, \reduce^\fraka)$, $\fraka \in \Id(\calO_K)$
  in $K$ at the same time, we saw that the set of all $f$-representations, $\fRep(K)$, can be
  identified with the (Arakelov) divisor class group $\Pic(K)/\G[\frakp_{n+1}]$ of $K$. This
  generalizes the result by Paulus and R\"uck \cite{paulus-rueck} for hyperelliptic function fields,
  and is compatible with the arithmetic in $\Pic^0(K)$ described by He\ss\ \cite{hessRR}. Moreover,
  our embedding of $\fRep(K)$ into $\Pic^0(K)$ in the number field case is similar to Schoof's
  embedding of $\Red(K)$ into the Arakelov divisor class group $\Pic^0(K)$.
  
  An important open question is how baby steps can be interpreted in our approach. One can interpret
  them as Buchmann in \cite{buchmann-habil} as a tool which computes all reduced elements whose
  distance lies in a given parallelepiped, as this allows baby-step giant-step algorithms for
  arbitrary infrastructures. Unfortunately, no efficient method for computing these ``baby steps''
  is known for number fields. Another open question is whether one can find an efficient unique
  representation of elements in $\Pic^0(K)$ in case no infinite place of degree~one is
  available. Having a unique representation of an element in $\Red(K)/_\sim$ is required to do fast
  look-ups as in algorithms of baby-step giant-step type.
  
  \section*{Acknowledgments}
  I would like to thank Michael J.~Jacobson and Hugh C.~Williams for introducing me into the subject
  of infrastructures, Andreas Stein for several discussions which gave me ideas and Renate Scheidler
  and Mark Bauer for discussions on certain aspects. Moreover, I would like to thank the
  mathematical institutes at the Universit\"at Oldenburg and the University of Calgary for their
  hospitality during my visits.  Finally, I would like to thank Renate Scheidler and the anonymous
  referee for their helpful and extensive comments and Rachel Suri for proof-reading.
  
\newcommand{\etalchar}[1]{$^{#1}$}
\providecommand{\bysame}{\leavevmode\hbox to3em{\hrulefill}\thinspace}
\providecommand{\MR}{\relax\ifhmode\unskip\space\fi MR }
\providecommand{\MRhref}[2]{%
  \href{http://www.ams.org/mathscinet-getitem?mr=#1}{#2}
}
\providecommand{\href}[2]{#2}


\begin{thebibliography}{CFA{\etalchar{+}}06}

\bibitem[AO82]{appelgateonishi}
H.~Appelgate and H.~Onishi, \emph{Periodic expansion of modules and its
  relation to units}, J. Number Theory \textbf{15} (1982), no.~3, 283--294.
  \MR{MR680533 (84d:12005)}

\bibitem[BBT94]{biehl-buchmann-thiel}
I.~Biehl, J.~A. Buchmann, and C.~Thiel, \emph{Cryptographic protocols based on
  discrete logarithms in real-quadratic orders}, Advances in Cryptology -
  CRYPTO '94, 14th Annual International Cryptology Conference, Santa Barbara,
  California, USA, August 21-25, 1994, Proceedings (Y.~Desmedt, ed.), Lecture
  Notes in Computer Science, vol. 839, Springer, 1994, pp.~56--60.

\bibitem[Ber63]{bergmannnetze}
G.~Bergmann, \emph{Theorie der {N}etze}, Mathematische Annalen \textbf{149}
  (1963), 361--418. \MR{MR0157948 (28 \#1176)}

\bibitem[BS05]{buchmann-schmidt}
J.~A. Buchmann and A.~Schmidt, \emph{Computing the structure of a finite
  abelian group}, Math. Comp. \textbf{74} (2005), no.~252, 2017--2026
  (electronic). \MR{MR2164109 (2006c:20108)}

\bibitem[Buc85]{genvoronoiIuII}
J.~A. Buchmann, \emph{A generalization of {V}orono\u\i 's unit algorithm. {I,
  II}}, J. Number Theory \textbf{20} (1985), no.~2, 177--209. \MR{MR790781
  (86g:11062a)}

\bibitem[Buc87a]{genlagrange}
\bysame, \emph{On the computation of units and class numbers by a
  generalization of {L}agrange's algorithm}, J. Number Theory \textbf{26}
  (1987), no.~1, 8--30. \MR{MR883530 (89b:11104)}

\bibitem[Buc87b]{buchmann-habil}
\bysame, \emph{Zur {K}omplexit\"at der {B}erechnung von {E}inheiten und
  {K}lassenzahl algebraischer {Z}ahlk\"orper}, Habilitationsschrift, October
  1987.

\bibitem[BW88]{buchmannwilliams-infrastructure}
J.~A. Buchmann and H.~C. Williams, \emph{On the infrastructure of the principal
  ideal class of an algebraic number field of unit rank one}, Math. Comp.
  \textbf{50} (1988), no.~182, 569--579. \MR{MR929554 (89g:11098)}

\bibitem[BW90]{buchmann-williams-qrkeyexchange}
\bysame, \emph{A key exchange system based on real quadratic fields (extended
  abstract)}, Advances in cryptology---CRYPTO '89 (Santa Barbara, CA, 1989),
  Lecture Notes in Comput. Sci., vol. 435, Springer, New York, 1990,
  pp.~335--343. \MR{MR1062244 (91f:94014)}

\bibitem[CFA{\etalchar{+}}06]{hehcc}
H.~Cohen, G.~Frey, R.~Avanzi, C.~Doche, T.~Lange, K.~Nguyen, and F.~Vercauteren
  (eds.), \emph{Handbook of elliptic and hyperelliptic curve cryptography},
  Discrete Mathematics and its Applications (Boca Raton), Chapman \& Hall/CRC,
  Boca Raton, FL, 2006. \MR{MR2162716 (2007f:14020)}

\bibitem[Coh96]{cohen}
H.~Cohen, \emph{A course in computational algebraic number theory}, third
  corrected ed., Graduate Texts in Mathematics, vol. 138, Springer-Verlag,
  Berlin, 1996. \MR{MR1228206 (94i:11105)}

\bibitem[DF64]{delonefaddeev}
B.~N. Delone and D.~K. Faddeev, \emph{The theory of irrationalities of the
  third degree}, Translations of Mathematical Monographs, Vol. 10, American
  Mathematical Society, Providence, R.I., 1964. \MR{MR0160744 (28 \#3955)}

\bibitem[Fon08]{ff-pohlighellman}
F.~Fontein, \emph{Groups from cyclic infrastructures and {P}ohlig-{H}ellman in
  certain infrastructures}, Adv. Math. Commun. \textbf{2} (2008), no.~3,
  293--307. \MR{MR2429459}

\bibitem[Fon09]{fontein-diss}
F.~Fontein, \emph{The infrastructure of a global field and baby step-giant step
  algorithms}, Ph.{D}. thesis, Universit\"at Z\"urich, March 2009,
  \verb'http://user.math.uzh.ch/fontein/diss-fontein.pdf'.

\bibitem[GHM08]{galbraithmirelesmorales-balanced}
S.~D. Galbraith, M.~Harrison, and D.~J. M{ireles Morales}, \emph{Efficient
  hyperelliptic arithmetic using balanced representation for divisors},
  Algorithmic Number Theory, 8th International Symposium, ANTS-VIII, Banff,
  Canada, May 17-22, 2008 (Berlin) (A.~J. van~der Poorten and A.~Stein, eds.),
  Lecture Notes in Computer Science, vol. 5011, Springer, 2008, pp.~342--356.

\bibitem[GPS02]{galbraith-paulus-smart}
S.~D. Galbraith, S.~M. Paulus, and N.~P. Smart, \emph{Arithmetic on
  superelliptic curves}, Math. Comp. \textbf{71} (2002), no.~237, 393--405
  (electronic). \MR{MR1863009 (2002h:14102)}

\bibitem[He{\ss}99]{hessDiss}
F.~He{\ss}, \emph{Zur {D}ivisorklassengruppenberechnung in globalen
  funktionenk\"orpern}, Ph.{D}. thesis, Technische Universit\"at Berlin, 1999.

\bibitem[He{\ss}02]{hessRR}
F.~He{\ss}, \emph{Computing {R}iemann-{R}och spaces in algebraic function
  fields and related topics}, J. Symbolic Comput. \textbf{33} (2002), no.~4,
  425--445. \MR{MR1890579 (2003j:14032)}

\bibitem[HP85]{hellegouarchpaysant}
Y.~Hellegouarch and R.~P{aysant-Le Roux}, \emph{Commas, points extr\'emaux et
  ar\^etes des corps poss\'edant une formule du produit}, C. R. Math. Rep.
  Acad. Sci. Canada \textbf{7} (1985), no.~5, 291--296. \MR{MR813946
  (87m:12011)}

\bibitem[HP87]{hellegouarchpaysant2}
\bysame, \emph{Invariants arithm\'etiques des corps poss\'edant une formule du
  produit; applications}, Ast\'erisque (1987), no.~147-148, 291--300, 345,
  Journ\'ees arithm\'etiques de Besan\c con (Besan\c con, 1985). \MR{MR891435
  (88k:11067)}

\bibitem[HP01]{huehnlein-paulus}
D.~H{\"u}hnlein and S.~M. Paulus, \emph{On the implementation of cryptosystems
  based on real quadratic number fields (extended abstract)}, Selected areas in
  cryptography (Waterloo, ON, 2000), Lecture Notes in Comput. Sci., vol. 2012,
  Springer, Berlin, 2001, pp.~288--302. \MR{MR1895598}

\bibitem[Jac99]{jacobson-phd}
M.~J. Jacobson, Jr., \emph{Subexponential class group computation in quadratic
  orders}, Ph.{D}. thesis, Technische Universit\"at Darmstadt, 1999.

\bibitem[JSS07]{JSS-CPoRHC}
M.~J. Jacobson, Jr., R.~Scheidler, and A.~Stein, \emph{Cryptographic protocols
  on real hyperelliptic curves}, Adv. Math. Commun. \textbf{1} (2007), no.~2,
  197--221. \MR{MR2306309 (2008c:94024)}

\bibitem[JSW01]{JSW-RQFKE}
M.~J. Jacobson, Jr., R.~Scheidler, and H.~C. Williams, \emph{The efficiency and
  security of a real quadratic field based key exchange protocol}, Public-key
  cryptography and computational number theory (Warsaw, 2000), de Gruyter,
  Berlin, 2001, pp.~89--112. \MR{MR1881630 (2003f:94062)}

\bibitem[JSW06]{JSW-RQFKE2}
\bysame, \emph{An improved real-quadratic-field-based key exchange procedure},
  Journal of Cryptology \textbf{19} (2006), no.~2, 211--239. \MR{MR2213407}

\bibitem[Lan09]{ericphd}
E.~Landquist, \emph{Infrastructure, {A}rithmetic, and {C}lass {N}umber
  {C}omputations in {P}urely {C}ubic {F}unction {F}ields of {C}haracteristic at
  {L}east $5$}, Ph.D. thesis, University of Illinois at Urbana-Champaign, 2009,
  \verb'http://www.math.uiuc.edu/~landquis/articles/landquist-thesis.pdf'.

\bibitem[Len82]{lenstra-infrastructure}
H.~W. Lenstra, \emph{On the computation of regulators and class numbers of
  quadratic fields}, Journ\'ees Arithm\'etiques 1980 (Exeter, 13th--19th April
  1980) (Cambridge) (J.~V. Armitage, ed.), London Mathematical Society Lecture
  Notes, no.~56, Cambridge University Press, 1982, pp.~123--150.

\bibitem[LSY03]{lee-scheidler-yarrish-cyclic}
Y.~Lee, R.~Scheidler, and C.~Yarrish, \emph{Computation of the fundamental
  units and the regulator of a cyclic cubic function field}, Experiment. Math.
  \textbf{12} (2003), no.~2, 211--225. \MR{MR2016707 (2004j:11143)}

\bibitem[Mau00]{maurer-phd}
M.~Maurer, \emph{Regulator approximation and fundamental unit computation for
  real-quadratic orders}, Ph.{D}. thesis, Technische Universit\"at Darmstadt,
  2000.

\bibitem[MST99]{mueller-stein-thiel}
V.~M{\"u}ller, A.~Stein, and C.~Thiel, \emph{Computing discrete logarithms in
  real quadratic congruence function fields of large genus}, Math. Comp.
  \textbf{68} (1999), no.~226, 807--822. \MR{MR1620235 (99i:11119)}

\bibitem[Neu99]{neukirch}
J.~Neukirch, \emph{Algebraic number theory}, Grundlehren der Mathematischen
  Wissenschaften [Fundamental Principles of Mathematical Sciences], vol. 322,
  Springer-Verlag, Berlin, 1999. \MR{MR1697859 (2000m:11104)}

\bibitem[PR99]{paulus-rueck}
S.~M. Paulus and H.-G. R{\"u}ck, \emph{Real and imaginary quadratic
  representations of hyperelliptic function fields}, Math. Comp. \textbf{68}
  (1999), no.~227, 1233--1241. \MR{MR1627817 (99i:11107)}

\bibitem[PWZ82]{pohst-zassenhaus-fundamentalunits3}
M.~Pohst, P.~Weiler, and H.~Zassenhaus, \emph{On effective computation of
  fundamental units. {II}}, Math. Comp. \textbf{38} (1982), no.~157, 293--329.
  \MR{MR637308 (83e:12005b)}

\bibitem[PZ77]{pohst-zassenhaus-fundamentalunits1}
M.~Pohst and H.~Zassenhaus, \emph{An effective number geometric method of
  computing the fundamental units of an algebraic number field}, Math. Comp.
  \textbf{31} (1977), no.~139, 754--770. \MR{MR0498486 (58 \#16595)}

\bibitem[PZ82]{pohst-zassenhaus-fundamentalunits2}
\bysame, \emph{On effective computation of fundamental units. {I}}, Math. Comp.
  \textbf{38} (1982), no.~157, 275--291. \MR{MR637307 (83e:12005a)}

\bibitem[Ros02]{rosen}
M.~Rosen, \emph{Number theory in function fields}, Graduate Texts in
  Mathematics, vol. 210, Springer-Verlag, New York, 2002. \MR{MR1876657
  (2003d:11171)}

\bibitem[SBW94]{scheidler-buchmann-williams-qrkeyexchange}
R.~Scheidler, J.~A. Buchmann, and H.~C. Williams, \emph{A key-exchange protocol
  using real quadratic fields}, J. Cryptology \textbf{7} (1994), no.~3,
  171--199. \MR{MR1286662 (96e:94015)}

\bibitem[Sch82]{schoof-infra}
R.~J. Schoof, \emph{Quadratic fields and factorization}, Computational methods
  in number theory, Part II, Math. Centre Tracts, vol. 155, Math. Centrum,
  Amsterdam, 1982, pp.~235--286. \MR{MR702519 (85g:11118b)}

\bibitem[Sch01]{scheidler-infrastructurepurelycubic}
R.~Scheidler, \emph{Ideal arithmetic and infrastructure in purely cubic
  function fields}, J. Th\'eor. Nombres Bordeaux \textbf{13} (2001), no.~2,
  609--631. \MR{MR1879675 (2002k:11209)}

\bibitem[Sch08]{schoofArakelov}
R.~J. Schoof, \emph{Computing {A}rakelov class groups}, MSRI Publications,
  vol.~44, pp.~447--495, Cambridge University Press, Cambridge, 2008.

\bibitem[Sha72]{shanks-infra}
D.~Shanks, \emph{The infrastructure of a real quadratic field and its
  applications}, Proceedings of the Number Theory Conference (Univ. Colorado,
  Boulder, Colo., 1972) (Boulder, Colo.), Univ. Colorado, 1972, pp.~217--224.
  \MR{MR0389842 (52 \#10672)}

\bibitem[SS98]{scheidler-stein-purelyunitrank1}
R.~Scheidler and A.~Stein, \emph{Unit computation in purely cubic function
  fields of unit rank 1}, Algorithmic number theory (Portland, OR, 1998)
  (Berlin), Lecture Notes in Comput. Sci., vol. 1423, Springer, 1998,
  pp.~592--606. \MR{MR1726104 (2000k:11145)}

\bibitem[SSW96]{SSW-KEiRQCFF}
R.~Scheidler, A.~Stein, and H.~C. Williams, \emph{Key-exchange in real
  quadratic congruence function fields}, Des. Codes Cryptogr. \textbf{7}
  (1996), no.~1-2, 153--174, Special issue dedicated to Gustavus J. Simmons.
  \MR{MR1377761 (97d:94009)}

\bibitem[Ste77]{steiner-units}
R.~Steiner, \emph{On the units in algebraic number fields}, Proceedings of the
  Sixth Manitoba Conference on Numerical Mathematics (Univ. Manitoba, Winnipeg,
  Man., 1976) (Winnipeg, Man.), Congress. Numer., XVIII, Utilitas Math., 1977,
  pp.~413--435. \MR{MR532716 (81b:12008)}

\bibitem[Ste92]{stein-da}
A.~Stein, \emph{Baby step-giant step-{V}erfahren in reellquadratischen
  {K}ongruenzfunktionenk\"orpern mit {C}harakteristik ungleich 2},
  Diplomarbeit, Universit\"at des Saarlandes, Saarbr\"ucken, 1992.

\bibitem[Ste97]{stein-realelliptic}
A.~Stein, \emph{Equivalences between elliptic curves and real quadratic
  congruence function fields}, J. Th\'eor. Nombres Bordeaux \textbf{9} (1997),
  no.~1, 75--95. \MR{MR1469663 (98d:11144)}

\bibitem[Sti93]{stichtenoth}
H.~Stichtenoth, \emph{Algebraic function fields and codes}, Universitext,
  Springer-Verlag, Berlin, 1993. \MR{MR1251961 (94k:14016)}

\bibitem[SW98]{stein-williams-regulator}
A.~Stein and H.~C. Williams, \emph{An improved method of computing the
  regulator of a real quadratic function field}, Algorithmic number theory
  (Portland, OR, 1998), Lecture Notes in Comput. Sci., vol. 1423, Springer,
  Berlin, 1998, pp.~607--620. \MR{MR1726105 (2000j:11201)}

\bibitem[SW99]{stein-williams-regulator2}
\bysame, \emph{Some methods for evaluating the regulator of a real quadratic
  function field}, Experiment. Math. \textbf{8} (1999), no.~2, 119--133.
  \MR{MR1700574 (2000f:11152)}

\bibitem[SZ91]{stein-zimmer}
A.~Stein and H.~G. Zimmer, \emph{An algorithm for determining the regulator and
  the fundamental unit of hyperelliptic congruence function field}, Proceedings
  of the 1991 International Symposium on Symbolic and Algebraic Computation,
  ISSAC '91, Bonn, Germany, July 15-17, 1991, Association for Computing
  Machinery, 1991, pp.~183--184.

\bibitem[Thi95]{thiel-comprep}
C.~Thiel, \emph{Short proofs using compact representations of algebraic
  integers}, J. Complexity \textbf{11} (1995), no.~3, 310--329. \MR{MR1349260
  (96d:11142)}

\bibitem[Vol03]{vollmer-phd}
U.~Vollmer, \emph{Rigorously analyzed algorithms for the discrete logarithm
  problem in quadratic number fields}, Ph.{D}. thesis, Technische Universit\"at
  Darmstadt, 2003.

\bibitem[Wil85]{williams-contfract-numbtheor-compus}
H.~C. Williams, \emph{Continued fractions and number-theoretic computations},
  Rocky Mountain J. Math. \textbf{15} (1985), no.~2, 621--655, Number theory
  (Winnipeg, Man., 1983). \MR{MR823273 (87h:11129)}

\bibitem[WW87]{williams-wunderlich}
H.~C. Williams and M.~C. Wunderlich, \emph{On the parallel generation of the
  residues for the continued fraction factoring algorithm}, Math. Comp.
  \textbf{48} (1987), no.~177, 405--423. \MR{866124 (88i:11099)}

\end{thebibliography}
\end{document}